\documentclass{article}

\usepackage[unicode=true]{hyperref}
\usepackage{amsmath}
\usepackage{amssymb}
\usepackage{amsthm}
\usepackage{dsfont}
\usepackage{cleveref}

\hypersetup{bookmarksnumbered}

\numberwithin{equation}{section}

\theoremstyle{plain}
\newtheorem{theorem}{Theorem}[section]
\newtheorem{proposition}[theorem]{Proposition}
\newtheorem{lemma}[theorem]{Lemma}
\newtheorem{corollary}[theorem]{Corollary}

\theoremstyle{definition}
\newtheorem{definition}[theorem]{Definition}

\theoremstyle{remark}
\newtheorem*{remark*}{Remark}

\binoppenalty=\maxdimen
\relpenalty=\maxdimen
\renewcommand{\emptyset}{\varnothing}

\newcommand{\N}{\mathbb{N}}
\newcommand{\Z}{\mathbb{Z}}

\newcommand{\R}{\mathbb{R}}
\newcommand{\CC}{\mathbb{C}}

\newcommand{\EE}{\mathbb{E}}
\newcommand{\PP}{\mathbb{P}}

\newcommand{\CalO}{\mathcal{O}}

\newcommand{\eps}{\varepsilon}

\newcommand{\FirstN}[1]{\underline{#1}}

\newcommand{\indicator}{\mathds{1}}
\newcommand{\supp}{\operatorname{supp}}
\newcommand{\Fourier}{\mathcal{F}}
\newcommand{\Schwartz}{\mathcal{S}}
\newcommand{\linspan}{\operatorname{span}}

\newcommand{\Alg}{\mathrm{Alg}}
\newcommand{\lebesgue}{\boldsymbol{\lambda}}

\newcommand{\n}{\mathbf{n}}
\newcommand{\m}{\mathbf{m}}
\newcommand{\X}{\mathbf{X}}

\newcommand{\avsum}{\mathop{\mathpalette\avsuminner\relax}\displaylimits}

\makeatletter
\newcommand\avsuminner[2]{%
  {\sbox0{$\m@th#1\sum$}%
   \vphantom{\usebox0}%
   \ooalign{%
     \hidewidth
     \smash{\vrule height\dimexpr\ht0+1pt\relax depth\dimexpr\dp0+1pt\relax}%
     \hidewidth\cr
     $\m@th#1\sum$\cr
   }%
  }%
}
\makeatother

\newcommand{\footremember}[2]{%
    \footnote{#2}
    \newcounter{#1}
    \setcounter{#1}{\value{footnote}}%
}

\makeatletter
\newcommand{\subjclass}[2][2020]{%
  \let\@oldtitle\@title%
  \gdef\@title{\@oldtitle\footnotetext{#1 \emph{Mathematics subject classification.} #2}}%
}
\newcommand{\keywords}[1]{%
  \let\@@oldtitle\@title%
  \gdef\@title{\@@oldtitle\footnotetext{\emph{Key words and phrases.} #1.}}%
}
\makeatother

\begin{document}
\title{$L^p$ sampling numbers\\ for the Fourier-analytic Barron space}
\author{
Felix Voigtlaender\footremember{KURML}{
KU Eichstätt--Ingolstadt, Mathematisch--Geographische Fakultät,
Lehrstuhl Reliable Machine Learning,
Ostenstraße 26,
85072 Eichstätt,
Germany.
Email: \href{mailto:felix.voigtlaender@ku.de}{\texttt{felix.voigtlaender@ku.de}}
}%
}
\keywords{Sampling;
Sampling numbers;
Rate of convergence;
Barron functions;
Curse of dimensionality;
Information based complexity}
\subjclass{94A20, 41A46, 46E15, 42B35, 41A25, 65D15, 41A63}

\maketitle

\begin{abstract}
  In this paper, we consider Barron functions $f : [0,1]^d \to \R$ of smoothness $\sigma > 0$,
  which are functions that can be written as
  \[
    f(x) = \int_{\R^d} F(\xi) \, e^{2 \pi i \langle x, \xi \rangle} \, d \xi
    \quad \text{with} \quad
    \int_{\R^d} |F(\xi)| \cdot (1 + |\xi|)^{\sigma} \, d \xi < \infty.
  \]
  For $\sigma = 1$, these functions play a prominent role in machine learning,
  since they can be efficiently approximated by (shallow) neural networks
  without suffering from the curse of dimensionality.

  For these functions, we study the following question:
  Given $m$ point samples $f(x_1),\dots,f(x_m)$ of an unknown Barron function $f : [0,1]^d \to \R$
  of smoothness $\sigma$, how well can $f$ be recovered from these samples,
  for an optimal choice of the sampling points and the reconstruction procedure?
  Denoting the optimal reconstruction error measured in $L^p$ by $s_m (\sigma; L^p)$,
  we show that
  \[
    m^{- \frac{1}{\max \{ p,2 \}} - \frac{\sigma}{d}}
    \lesssim s_m(\sigma;L^p)
    \lesssim (\ln (e + m))^{\alpha(\sigma,d) / p}
             \cdot m^{- \frac{1}{\max \{ p,2 \}} - \frac{\sigma}{d}}
    ,
  \]
  where the implied constants only depend on $\sigma$ and $d$ and where
  $\alpha(\sigma,d)$ stays bounded as $d \to \infty$.
\end{abstract}

\section{Introduction}%
\label{sec:Introduction}

In \cite{BarronUniversalApproximationBoundsForSigmoidal,BarronApproximationAndEstimationBoundsForNN},
Andrew Barron introduced a certain function class---nowadays called the class of Barron functions---%
which has the fascinating property that Barron functions can be well approximated
(without suffering from the curse of dimensionality) by \emph{non-linear} function classes
like neural networks, whereas any linear method of approximating Barron functions
will be subject to the curse of dimensionality \cite{BarronUniversalApproximationBoundsForSigmoidal}.

Specifically, a function $f : [0,1]^d \to \R$ is called a Barron function if there exists
$F : \R^d \to \CC$ satisfying
\begin{equation}
  f(x) = \int_{\R^d} F(\xi) \, e^{2 \pi i \langle x, \xi \rangle} \, d \xi
  \qquad \forall \, x \in [0,1]^d
  \label{eq:IntroductionBarronRepresentation}
\end{equation}
and $C_f := \int_{\R^d} |F(\xi)| \cdot (1 + |\xi|) \, d \xi < \infty$.
For such a function, Barron showed in \cite{BarronUniversalApproximationBoundsForSigmoidal}
that for each $n \in \N$, there exists a shallow neural network $\Phi_n$
with a sigmoidal activation function and $n$ neurons satisfying
$\| f - \Phi_n \|_{L^2} \lesssim C_f \cdot n^{-1/2}$, where it is noteworthy that the
rate of decay of the approximation error is independent of the input dimension $d$.

As a consequence, in \cite{BarronApproximationAndEstimationBoundsForNN} Barron established
a learning bound for Barron functions.
More precisely, he showed that there is a reconstruction procedure based on neural networks,
such that given $m$ \emph{random} point samples $f(x_1),\dots,f(x_m)$ where $x_1,\dots,x_m$
are independent and uniformly distributed in $[0,1]^d$, then the resulting reconstruction
$\widetilde{f} = \widetilde{f}_{x_1,\dots,x_m}$ satisfies
$\EE \| f - \widetilde{f} \|_{L^2} \lesssim \bigl( \ln(m) / m \bigr)^{1/4}$, where the expectation is
with respect to the choice of the sampling points $x_1,\dots,x_m$.
Similar results with a worse asymptotic scaling for $m \to \infty$ but with
better dependence of the implied constant on $d$ were recently derived
in \cite{KlusowskiBarronHighDimensionalRidgeFunctionCombinations}.

As for the approximation result, an important property of these learning bounds is
that the resulting asymptotic decay rate of the reconstruction error
is \emph{independent of the ambient dimension}.
Nevertheless, as recently noted in \cite[Page~23]{DeVoreOptimalLearning},
the result in \cite{BarronApproximationAndEstimationBoundsForNN} leaves open
the following question, which---once resolve---has important implications for the
optimal learning of Barron functions:
\begin{equation}
  \tag{$\ast$}\label{eq:IntroductionMainQuestion}
  \parbox{\dimexpr\linewidth-4em}{%
    \strut
    If one is free to choose the sampling points $x_1,\dots,x_m \in [0,1]^d$
    and the reconstruction method, what is the optimal behavior of the worst case
    reconstruction error as $m \to \infty$?
    \strut
  }
\end{equation}

In this paper, we resolve this question, by deriving upper and lower
bounds for the so-called \emph{sampling numbers} of the Barron space.
The derived bounds are sharp up to logarithmic factors.

Before we dive into the precise results, let us define the precise notion
of Barron functions and sampling numbers that we will use.

\subsection{Definition of the Barron space and sampling numbers}%
\label{sub:IntroductionDefinitions}

In this paper, we will study a certain generalization of the Barron spaces
that was recently introduced in \cite{KlusowskiBarronHighDimensionalRidgeFunctionCombinations}.
The original Barron space studied in
\cite{BarronUniversalApproximationBoundsForSigmoidal,BarronApproximationAndEstimationBoundsForNN}
corresponds to choosing $\sigma = 1$.

\begin{definition}\label{def:BarronFunctions}
  Let $\sigma > 0$ and let $\Omega \subset \R^d$ be arbitrary.
  We say that $f : \Omega \to \R$ is a \emph{Barron function} with smoothness parameter $\sigma$
  (written $f \in B^\sigma (\Omega)$), if there exists a measurable function $F : \R^d \to \CC$
  satisfying
  \(
    \int_{\R^d}
      (1 + |\xi|)^\sigma \, |F(\xi)|
    \, d \xi
    < \infty
  \)
  and
  \begin{equation}
    \forall \, x \in \Omega: \qquad
      f(x) = \int_{\R^d}
               e^{2 \pi i \langle x, \xi \rangle} \, F(\xi)
             \, d \xi
    .
    \label{eq:BarronRepresentation}
  \end{equation}
  We then define the \emph{Barron norm} of $f$ as
  \[
    \| f \|_{B^\sigma (\Omega)}
    := \inf
       \bigg\{
         \int_{\R^d} (1 + |\xi|)^\sigma |F(\xi)| \, d \xi
         \,\,\colon\,\,
         \text{\Cref{eq:BarronRepresentation} holds}
       \bigg\}
    .
  \]
\end{definition}

\begin{remark*}
  To simplify some notations, in the following we will work on the unit cube
  $[-\frac{1}{2}, \frac{1}{2}]^d$ instead of on $[0,1]^d$.
  Since the Barron space is translation invariant (see e.g.\ \Cref{lem:TranslationDilation})
  this does not change the results.
\end{remark*}

The question \eqref{eq:IntroductionMainQuestion} from above
is formalized by the following definition of \emph{sampling numbers}.
For more general properties of such sampling numbers,
we refer to \cite[Sections~2 and 3]{HeinrichRandomApproximation}.

\begin{definition}\label{def:SamplingNumbers}
  Let $\Omega \subset \R^d$ be arbitrary and let $V \hookrightarrow C(\Omega)$ be a normed vector space.
  Furthermore, let $R$ (the ``reconstruction space'') be a normed vector space of measurable
  functions on $\Omega$ for which $V \hookrightarrow R$.

  For $m \in \N$, the $m$-th \emph{(non-linear) sampling number} of $V$ is defined as
  \[
    s_m (V; R)
    := \inf_{x_1,\dots,x_m \in \Omega} \,\,
         \inf_{T : \R^m \to R} \,\,
           \sup_{f \in V, \| f \|_V \leq 1} \,\,
             \| f - T(f(x_1),\dots,f(x_m)) \|_{R}
    ,
  \]
  where the function $T : \R^m \to R$ can be non-linear.
  The \emph{linear sampling number} $s_m^{\mathrm{lin}}(V; R)$ is defined similarly,
  but assuming $T$ to be linear.
\end{definition}

Strictly by definition, having a small sampling number just means that functions
in the unit ball of $V$ can be well recovered via samples.
The following result shows that then actually \emph{any} function $f \in V$
can be well reconstructed (compared to its norm).
This is probably well-known in the information-based complexity literature,
but since we were not able to locate a convenient reference,
we provide the proof in \Cref{sub:ReconstructionOnWholeSpaceProof}.

\begin{lemma}\label{lem:ReconstructionOnWholeSpace}
  For any $m \!\in\! \N$ and $\delta \!>\! 0$, there exist sampling points $x_1,...,x_m \!\in\! \Omega$
  and a map $A : \R^m \to V$ satisfying
  \[
    \big\| f - A(f(x_1),\dots,f(x_m)) \big\|_R
    \leq (4 + \delta) \cdot s_m(V; R) \cdot \| f \|_{V}
    \qquad \forall \, f \in V.
  \]
\end{lemma}

\subsection{Main result}%
\label{sub:IntroductionMainResult}

Our main result is to provide (essentially) sharp bounds for the sampling numbers of the Barron space
$B^{\sigma}(\Omega)$, as stated in the following theorem.

\begin{theorem}\label{thm:IntroductionMainResult}
  Let $d \in \N$ and $\sigma \in (0,\infty)$, and define $\Omega := [-\frac{1}{2}, \frac{1}{2}]^d$.
  Then there exist constants $C_1 = C_1 (d,\sigma) > 0$ and $C_2 = C_2(d,\sigma) > 0$
  satisfying for any $p \in [1,\infty]$ and $m \in \N$ that
  \[
    s_m \bigl( B^\sigma(\Omega); L^p(\Omega) \bigr)
    \leq \begin{cases}
           C_1 \cdot \big( [\ln(e + m)]^4 \big/ m \big)^{\frac{1}{2} + \frac{\sigma}{d}} ,
           & \text{if } p \in [1,2], \\[0.2cm]
           C_1 \cdot [\ln(e + m)]^{(4 + 8 \frac{\sigma}{d}) / p} \cdot m^{-(\frac{1}{p} + \frac{\sigma}{d})} ,
           & \text{if } p \in [2,\infty]
         \end{cases}
  \]
  and
  \[
    s_m \bigl( B^\sigma(\Omega); L^p(\Omega) \bigr)
    \geq \begin{cases}
           C_2 \cdot m^{-(\frac{1}{2} + \frac{\sigma}{d})},
           & \text{if } p \in [1,2], \\[0.2cm]
           C_2 \cdot m^{-(\frac{1}{p} + \frac{\sigma}{d})},
           & \text{if } p \in [2,\infty] .
         \end{cases}
  \]
\end{theorem}

\begin{remark*}
  1) Up to log factors, the preceding bound shows that
  \[
    s_m \bigl( B^\sigma(\Omega); L^p(\Omega) \bigr)
    \asymp m^{-(\frac{1}{ \max \{ p, 2 \} } + \frac{\sigma}{d})}
    .
  \]
  Thus, we see for any fixed $p \in [1,\infty)$ that---at least regarding the asymptotic rate---%
  \emph{Barron functions can be reconstructed from point samples
  without incurring the curse of dimensionality}.
  Note, however, that this only holds if $p \in [1,\infty)$ is fixed \emph{independent of $d$}.
  On the contrary, this is not at all true if one measures the reconstruction error in $L^\infty$
  (unless $\sigma \gtrsim d$).

  \medskip{}

  2) We emphasize that \emph{the exponent of the logarithmic factor stays bounded as $d \to \infty$}.

  \medskip{}

  3) In this paper we do \emph{not} study the asymptotic behavior
  of the constants $C_i = C_i (d,\sigma)$ with respect to $d$.
  Thus, it could be that the (essentially dimension-free) asymptotic rate only takes
  effect once $m \gg m_0 = m_0(d)$ with $m_0(d)$ potentially growing quickly with $d$.

  \medskip{}

  4) As stated above, the lower bound only applies to \emph{deterministic} methods of reconstructing
  Barron functions from point samples.
  In fact, the same bound also holds for \emph{randomized} algorithms;
  see \Cref{sec:LowerBounds} and in particular \Cref{thm:LowerBound} for more details.

  \medskip{}

  5) The reconstruction algorithm that we use to prove the upper bound is \emph{non-linear},
  since it is based on ($\ell^1$-minimization) methods from compressive sensing
  \cite{RauhutFoucartIntroToCS}.
  In fact, we show that---at least if one measures the reconstruction error in $L^p$ with $p \geq 2$---%
  any linear method of reconstruction will necessarily suffer from the curse of dimensionality;
  see \Cref{thm:LinearLowerBound} for more details.
\end{remark*}

\begin{proof}
  The upper bound will be shown in \Cref{thm:GeneralUpperBound}.
  The lower bound will be shown in \Cref{thm:LowerBound}.
\end{proof}

\subsection{Discussion of related work}%
\label{sub:RelatedWork}

\subsubsection{Results related to Barron spaces}

After their inception in
\cite{BarronUniversalApproximationBoundsForSigmoidal,BarronApproximationAndEstimationBoundsForNN},
the Barron space and its relation to neural networks has recently received much interest due to the
increased interest in rigorously understanding the astonishing practical success
of \emph{deep learning} methods:
In \cite{KlusowskiBarronApproximation}, it was shown that functions belonging to the Barron spaces
$B^2(\Omega)$ can be approximated up to $L^\infty$ error $\CalO(m^{-\frac{1}{2} - \frac{1}{d}})$
by shallow ReLU neural networks satisfying an $\ell^1$ bound on the network weights;
similar results were derived for the approximation of functions in $B^3 (\Omega)$
by neural networks with the squared ReLU activation function.
In \cite{KlusowskiBarronHighDimensionalRidgeFunctionCombinations}, Klusowski and Barron
proved learning bounds roughly similar to those in \cite{BarronApproximationAndEstimationBoundsForNN},
with a worse asymptotic rate of decay but much milder dependence of the implied constants
on the input dimension.
Pointwise smoothness properties of Barron functions were studied in
\cite{WeinanERepresentationFormulasForBarronFunctions}.

In \cite{CarageaPetersenVoigtlaenderNNApproximationBarronBoundary,PetersenVoigtlaenderOptimalLearningBarronBoundary},
the authors studied (optimal) learning bounds for classification functions
with decision boundary in the Barron space.
These learning bounds are intimately related to certain functional-analytic properties
of the Barron space, namely the behavior of the \emph{entropy numbers} of this space.
In \cite[Propositions~4.4 and 4.6]{PetersenVoigtlaenderOptimalLearningBarronBoundary},
it was shown that the $L^p$ entropy numbers of $B^1([0,1]^d)$ for arbitrary $p \in [1,\infty]$ satisfy
\begin{equation}
  \eps^{- 1 / (\frac{1}{2} + \frac{1}{d})}
  \lesssim M_{B^1, L^p} (\eps)
  \lesssim \eps^{- 1 / (\frac{1}{2} + \frac{1}{d})} \cdot \bigl(1 + \ln(1 / \eps)\bigr)
  \qquad \forall \, \eps \in (0,1)
  .
  \label{eq:IntroductionBarronEntropyNumbers}
\end{equation}
A similar bound (valid for general $\sigma > 0$)
appeared almost concurrently in \cite{SiegelUniformApproximationRatesAndNWidthsOfShallowNN};
see also \cite{SiegelEntropyAndWidthOfShallowNN} for related bounds
for a somewhat different (but related) function class of neural networks.

A surprising property of our results regarding the sampling numbers of the Barron spaces
is that while the asymptotic behavior of the \emph{entropy numbers} of $B^\sigma$
are independent (up to a log factor) of the choice of $p$, the sampling numbers vary greatly
with $p$; as noted above, the $L^\infty$ sampling numbers of $B^\sigma$ are subject to the curse
of dimensionality, while this is not true of the entropy numbers.
We remark that there is a general lower bound for the so-called \emph{Gelfand numbers}
(which are a lower bound for the sampling numbers) in terms of the entropy numbers,
known as \emph{Carl's inequality}; see e.g.\ \cite{CarlsInequality}.
Thus, our bounds show that Carl's inequality is not sharp in this case.
We also remark that in several other settings (see for instance the papers
\cite{UllrichNewUpperBoundForSamplingNumbers,UllrichRKHSSamplingRecoveryLInfty},
which study the sampling numbers for certain reproducing kernel Hilbert spaces)
there indeed is a very close relationship between the behavior of the entropy numbers
and the sampling numbers.

\subsubsection{Results related to sampling numbers}

The relation between sampling numbers and optimal learning bounds
was recently emphasized in \cite{DeVoreOptimalLearning}, in which also the
question \eqref{eq:IntroductionMainQuestion} studied in this paper was raised.
The main purpose of the article \cite{DeVoreOptimalLearning} is to devise
optimization problems which (at least in principle) can be numerically solved
and which attain the theoretically optimal recovery rate.
We remark that the proof of \Cref{thm:IntroductionMainResult} actually indicates
a recovery algorithm for the case of Barron spaces: one locally reconstructs
the function of interest via solving certain $\ell^1$ minimization problems
and then ``patches together'' these local reconstructions.

For the classical Sobolev spaces $W^{k,p}([0,1]^d)$, the asymptotic behavior
of the sampling numbers is well-studied.
First, note (in the case of $d > 1$) that $W^{k,p}([0,1]^d) \hookrightarrow C([0,1]^d)$
if and only if $k/d > 1/p$.
In this case, \cite[Theorem~6.1]{HeinrichRandomApproximation} shows that
\begin{equation}
  s_m \bigl(W^{k,p}([0,1]^d); L^q(\Omega)\bigr)
  \asymp \begin{cases}
           m^{-k/d},             & \text{if } q \leq p, \\
           m^{-k/d + 1/p - 1/q}, & \text{if } q > p,
         \end{cases}
  \label{eq:SobolevSamplingNumbers}
\end{equation}
and the same bound also holds for randomized recovery algorithms,
meaning that allowing potentially randomized algorithms does not improve the recovery.

The question answered by the sampling numbers is, how well an unknown function
from a given class can be recovered, if one has free choice over the sampling points
and over the recovery algorithm.
However, in many scenarios (for example, in many settings studied in machine learning),
it is unrealistic to assume that the samples can be chosen freely;
rather, they are often sampled from a given probability distribution.
For the setting of Sobolev functions,
it was recently shown in \cite{KriegSonnleitnerRandomPointsForApproximatingSobolevFunctions}
that in fact, \emph{randomly chosen points (with respect to the uniform distribution)
perform as good as the optimal points}
(at least regarding the asymptotic behavior of the error for $m \to \infty$).
It would be interesting to study if this remarkable property also holds for the Barron spaces.

One limitation of the rates in \Cref{eq:SobolevSamplingNumbers} is that they are subject
to the \emph{curse of dimensionality}: Unless the smoothness $k$ is very large
(meaning $k \gtrsim d$), the decay of the recovery error with the number of samples
is extremely slow for large input dimension $d$.
For this reason, much research has been concerned with finding function spaces that
do \emph{not} suffer from this curse of dimensionality but which are nevertheless
sufficiently rich that functions of interest in practice belong to such spaces.
In addition to the Barron spaces, one further class of spaces with this property
are the \emph{spaces of dominating mixed smoothness};
see \cite{DungTemlyakovUllrichHyperbolicCrossApproximation} for a good overview.
The main issue with these spaces is that while the ``main rate of decay''
for quantities like the entropy numbers or (linear) sampling numbers is dimension-free,
one gets a further power of a logarithmic term, and the power grows unboundedly as $d \to \infty$.
As a case in point, \cite[Theorem~5.12]{DungTemlyakovUllrichHyperbolicCrossApproximation} shows
for $r > \frac{1}{2}$ that
\begin{equation}
  s_m^{\mathrm{lin}} (\mathbf{W}_2^r; L^\infty)
  \asymp m^{-(r - 1/2)} \cdot (\ln m)^{(d-1) \cdot r}
  ,
  \label{eq:DominatingMixedSmoothnessSamplingNumber}
\end{equation}
where $\mathbf{W}_2^r$ is a Sobolev space of dominating mixed smoothness.
It should be remarked that this is only for \emph{linear} recovery algorithms,
which in the case of the Barron space are fully subject to the curse of dimensionality
(see \Cref{sec:LinearSamplingNumbersLowerBound}).
I do not know whether similar results hold for the \emph{non-linear} sampling numbers,
but I deem it unlikely, since also the entropy numbers of the spaces of dominating mixed smoothness
are subject to such a logarithmic factor;
see \cite[Theorem~6.13]{DungTemlyakovUllrichHyperbolicCrossApproximation}.
We emphasize that while our upper bounds for the sampling numbers for the Barron space
involve a logarithmic factor as well, the power of the logarithmic factor
stays bounded as $d \to \infty$, in stark contrast to \Cref{eq:DominatingMixedSmoothnessSamplingNumber}.

Finally, we would also like to mention the recent work \cite{GrohsVoigtlaenderNNSamplingNumbers}
which studies the sampling numbers of certain function spaces related to neural networks,
namely so-called neural network approximation spaces.

\subsection{Proof techniques and structure of the paper}%
\label{sub:ProofTechniques}

In \Cref{sec:UpperBounds}, we prove the upper bound in \Cref{thm:IntroductionMainResult}.
The case $p = \infty$ is handled simply by noting that $B^\sigma (\Omega) \hookrightarrow C^s(\Omega)$
and then using known results for the sampling numbers of $C^s(\Omega)$.
Then, once the case $p = 2$ is settled as well, the case $p \in [2,\infty]$ follows by
a form of interpolation (see \Cref{lem:SpecialInterpolationLemma}), and the case $p \in [1,2]$
follows from the embedding
$L^2([-\frac{1}{2},\frac{1}{2}]^d) \hookrightarrow L^p([-\frac{1}{2}, \frac{1}{2}]^d)$
which holds for $p \in [1,2]$.
Thus, the main difficulty is the case $p = 2$.
The proof of this case proceeds as follows:
\begin{itemize}
  \item We first prove a \emph{local} reconstruction result, where we show that if
        we can sample $f \in B^\sigma([-\frac{1}{2}, \frac{1}{2}]^d)$ on all of
        $\Omega = [-\frac{1}{2},\frac{1}{2}]^d$, then we can reconstruct $f$
        on the \emph{smaller} domain $\Omega_0 = [-\frac{1}{4},\frac{1}{4}]^d$
        with an error bound similar to the one stated in \Cref{thm:IntroductionMainResult}.

        This is shown by first proving that if $f \in B^{\sigma}(\Omega)$ and if $\varphi$
        is a smooth cutoff function satisfying $\supp \varphi \subset (-\frac{1}{2}, \frac{1}{2})^d$
        and $\varphi \equiv 1$ on $\Omega_0$, then the localized function $\varphi \cdot f$
        admits a Fourier expansion with coefficients in a weighted $\ell^1$ space.
        This then implies (by a deep result from \cite{DeVoreNonLinearApproximationByTrigonometricSum})
        that $\varphi \cdot f$ can be well approximated by a \emph{sparse} trigonometric polynomial
        $g$.
        By interpreting the samples of $f \cdot \varphi$ as noisy samples of $g$,
        we then use results from compressive sensing to recover an approximation $h$ of $g$,
        which is then a good approximation of $f \cdot \varphi$.

        Note that considering the localized function $f \cdot \varphi$ is really necessary:
        Since $f$ need not be periodic, $f$ itself will not have a Fourier expansion
        with well-controlled coefficients.

  \item We extend the \emph{local} reconstruction result to a global result, by using
        a suitable partition $\Omega = \biguplus_{n \in \N_0^d, \theta \in \{ \pm \}^d} Q_{\n}^{\theta}$
        of $\Omega$ into $d$-dimensional rectangles $Q_{\n}^\theta$ with the property that
        there is a larger rectangle $P_{\n}^\theta \supset Q_{\n}^\theta$
        that is still contained in $\Omega$ and such that (after a suitable affine change of variables),
        $Q_{\n}^\theta$ lies with respect to $P_{\n}^\theta$ as $\Omega_0$ lies with respect to $\Omega$.
        After a change of coordinates, this essentially allows one to sample $f|_{Q_{\n}^\theta}$
        on the bigger set $P_{\n}^\theta \subset \Omega$ and then to employ
        the local reconstruction result to reconstruct an approximation to $f|_{Q_{\n}^\theta}$.
        By ``patching together'' these local results and by suitably choosing
        the number of samples used on each $Q_{\n}^\theta$, one then obtains the global result.
\end{itemize}

\Cref{sec:LowerBounds} proves the lower bound in \Cref{thm:IntroductionMainResult}.
This is based on showing that there is a large family of sums of bump functions
contained in the Barron space $B^{\sigma}(\Omega)$.
Then, given the position of the sampling points, one can choose the sum of bump functions $f$
such that it avoids all the sampling points (and is thus indistinguishable from the zero function),
but such that $f$ still has comparatively big $L^p$ norm.
This proves the lower bound for deterministic recovery procedures.
The extension to possibly randomized algorithms is based on proving a lower bound
for the \emph{average case} performance of any deterministic algorithm
over a suitably chosen class of sums of bump functions.

Finally, in \Cref{sec:LinearSamplingNumbersLowerBound}, we show that \emph{linear} sampling
recovery procedures for the Barron space are necessarily subject to the curse of dimensionality.
This is based on observations by Barron in \cite[Section~6]{BarronUniversalApproximationBoundsForSigmoidal},
which we slightly adapt in order to also cover the case of \emph{randomized} linear recovery procedures.

Some more technical proofs are postponed to \Cref{sec:Appendix}.

\subsection{Notation}%
\label{sub:Notation}

We denote the standard inner product of $x,y \in \R^d$
by $\langle x,y \rangle = \sum_{j=1}^{d} x_j y_j$.
The Euclidean norm $\| x \|_{\ell^2} = \sqrt{\langle x,x \rangle}$
of $x \in \R^d$ is denoted by $|x| = \| x \|_{\ell^2}$.
We denote the Lebesgue measure of a (measurable) set $M \subset \R^d$ by $\lebesgue(M)$.
Given $a = (a_1,\dots,a_d) \in \R^d$, we denote by $D_a = \mathrm{diag}(a_1,\dots,a_d)$
the diagonal matrix with diagonal entries equal to $a_1,\dots,a_d$.

For the Fourier transform, we use the normalization
\[
  \Fourier f (\xi)
  = \widehat{f}(\xi)
  = \int_{\R^d}
      f(x) \cdot e^{-2 \pi i \langle x, \xi \rangle}
    \, d x
\]
for $f \in L^1(\R^d)$.
It is well-known that $\Fourier$ extends to a unitary operator on $L^2 (\R^d)$
and that the inverse Fourier transform is given by $\Fourier^{-1} f (x) = \Fourier f(-x)$.

Given a finite set $\Theta \neq \emptyset$ and any sequence $(a_\theta)_{\theta \in \Theta} \subset \R$,
we write
\[
  \avsum_{\theta \in \Theta} a_\theta
  := \frac{1}{|\Theta|}
     \sum_{\theta \in \Theta}
       a_\theta
\]
for the average of the $a_\theta$.

\section{Upper bounds}%
\label{sec:UpperBounds}

In this section, we derive (almost) optimal upper bounds for the sampling numbers
of the Barron space $B^{\sigma}(\Omega)$ for $\Omega = [-\frac{1}{2}, \frac{1}{2}]^d$,
where the reconstruction error is measured in $L^p$.
We start with the case $p = 2$ (\Cref{sub:L2UpperBounds}),
then consider the case $p = \infty$ (\Cref{sub:LInftyUpperBounds}),
and then use these two cases to obtain bounds for the general case $p \in [1,\infty]$.

\subsection{Upper bounds for \texorpdfstring{$L^2$}{L²} sampling numbers}%
\label{sub:L2UpperBounds}

We first note the following elementary embedding of the Barron space into $L^\infty$.

\begin{lemma}\label{lem:BarronEmbedsIntoLInfty}
  Let $\Omega \subset \R^d$ and $\sigma \in (0,\infty)$ be arbitrary.
  Then every $f \in B^\sigma(\Omega)$ satisfies
  \[
    \sup_{x \in \Omega} |f(x)|
    \leq \| f \|_{B^\sigma(\Omega)}
    .
  \]
\end{lemma}

\begin{proof}
  Let $f \in B^\sigma(\Omega)$.
  By definition of the norm $\| \bullet \|_{B^\sigma(\Omega)}$, given any $\delta > 0$,
  there exists $F : \R^d \to \CC$ satisfying
  $\int_{\R^d} (1 + |\xi|)^\sigma |F(\xi)| \, d \xi \leq \| f \|_{B^\sigma(\Omega)} + \delta$
  and such that $f(x) = \int_{\R^d} F(\xi) e^{2 \pi i \langle x, \xi \rangle}\, d \xi$
  for all $x \in \Omega$, which then easily implies $|f(x)| \leq \| f \|_{B^\sigma(\Omega)} + \delta$
  for all $x \in \Omega$.
  Since $\delta > 0$ was arbitrary, this implies the stated estimate.
\end{proof}

The following is the main technical lemma on which our proof relies.
It is based on proof techniques in
\cite[Lemmas~4.3 and 4.5]{PetersenVoigtlaenderOptimalLearningBarronBoundary}.

\begin{lemma}\label{lem:ReconstructionOnSmallerSet}
  Let $d \in \N$ and $\sigma > 0$ be given
  and let $\Omega := [-\frac{1}{2}, \frac{1}{2}]^d$
  and $\Omega_0 := [-\frac{1}{4}, \frac{1}{4}]^d$.
  There exists a constant $\kappa = \kappa(d,\sigma) \geq 1$
  and an absolute constant $C \geq 1$ such that if $\eps > 0$ and if
  \begin{equation}
    m
    \geq \kappa
         \cdot \eps^{- 1 / (\frac{1}{2} + \frac{\sigma}{d})}
         \cdot \ln^4 \bigl( e + \eps^{-1} \bigr)
    ,
    \label{eq:AssumptionNumberOfSamples}
  \end{equation}
  then there exist
  \[
    \text{sampling points } x_1,\dots,x_m \in \Omega
    \text{ and }
    \text{a reconstruction map } T : \R^m \to L^2(\Omega)
  \]
  such that the following holds:
  \[
    \forall \, f \in B^\sigma (\Omega) \text{ with } \| f \|_{B^\sigma(\Omega)} \leq 1: \quad
      \| f - T(f(x_1),\dots,f(x_m)) \|_{L^2(\Omega_0)}
    \leq C \cdot \eps .
  \]
\end{lemma}

\begin{remark*}
  Note that samples $x_i$ come from all of $\Omega$
  but that the reconstruction is only on $\Omega_0$, not on all of $\Omega$.
\end{remark*}

\begin{proof}
  For $f \in B^\sigma(\Omega)$ with $\| f \|_{B^\sigma(\Omega)} \leq 1$,
  \Cref{lem:BarronEmbedsIntoLInfty} shows that $\| f \|_{L^\infty(\Omega)} \leq 1$.
  Thus, in case that $\eps \geq 1$, we can even choose $m = 0$ and $T : \R^m \to L^2(\Omega)$
  such that $T \equiv 0$.
  Therefore, in the following we need only consider the case $\eps < 1$.

  \medskip{}

  We denote the Fourier basis for $L^2(\Omega)$ by
  $(e_k)_{k \in \Z} := (e^{2 \pi i \langle k,\bullet \rangle})_{k \in \Z}$,
  and for any $g \in L^1(\Omega)$, we denote the associated discrete Fourier coefficients by
  $\widehat{f}(k) = \int_{\Omega} f(x) \, \overline{e_k(x)} \, d x$.
  Furthermore, we identify $L^p (\Omega) \cong L^p (\R^d / \Z^d)$ for any $p \in [1,\infty]$,
  meaning that we identify a function $f \in L^p(\Omega)$ with its $\Z^d$-periodic extension
  to $\R^d$, so that we can sensibly talk about the convolution $f \ast g$ of two functions
  $f,g \in L^1(\Omega)$.
  It will be clear from context in each instance whether we consider
  the Fourier transform or convolution on $\Omega$ or on $\R^d$.

  \medskip{}

  \textbf{Step~1 (Preparations):}
  We fix a function $\varphi \!\in\! C_c^\infty (\R^d)$ satisfying $0 \leq \varphi \leq 1$
  and $\supp \varphi \subset \Omega^{\circ}$ and such that $\varphi \equiv 1$ on $\Omega_0$.
  For brevity, let us define $w : \R^d \to (0,\infty), \xi \mapsto (1 + |\xi|)^\sigma$.
  We define the weighted $L^1$ space $L_w^1$ as the set of all measurable $G : \R^d \to \CC$
  for which the norm $\| G \|_{L_w^1} := \| G_w \|_{L^1}$ is finite, where we use the notation
  \[
    G_w : \quad
    \R^d \to \CC, \quad
    G_w (\xi) = w(\xi) \cdot G(\xi)
    .
  \]
  We note for later use for $\xi, \eta \in \R^d$ that
  \(
    1 + |\xi + \eta|
    \leq 1 + |\xi| + |\eta|
    \leq (1 + |\xi|) \cdot (1 + |\eta|) ,
  \)
  and raising this to the power $\sigma$ shows that $w$ is submultiplicative, that is,
  \begin{equation}
    w(\xi + \eta)
    \leq w(\xi) \cdot w(\eta)
    \qquad \forall \, \xi, \eta \in \R^d .
    \label{eq:WeightSubmultiplicative}
  \end{equation}

  Now, given any continuous function $G : \R^d \to \CC$,
  we define the (Wiener) maximal function of $G$ as
  \[
    M G : \quad
    \R^d \to [0,\infty), \quad
    \xi \mapsto \sup_{\eta \in \xi + [-1/2, 1/2]^d} |G(\eta)|
    .
  \]
  It is easy to see%
  \footnote{Indeed, if $\lambda \in [0,\infty)$ and $M G (\xi) > \lambda$, then there exists
  $\eta \in \xi + [-1/2, 1/2]^d$ such that $|G(\eta)| > \lambda$.
  By continuity of $G$, there then exists $\eps > 0$ such that $|G(\eta + v)| > \lambda$
  for all $v$ with $\| v \| < \eps$.
  Now, if $\widetilde{\xi} = \xi + w$ with $\| w \| < \eps$, then $\eta + w \in (\xi + w) + [-1/2,1/2]^d$
  and thus $M G (\widetilde{\xi}) \geq |G(\eta + w)| > \lambda$.
  This shows that $\{ \xi \in \R^d \colon M G (\xi) > \lambda \}$ is open for every $\lambda \in [0,\infty)$,
  and this implies that $M G$ is measurable.
  }
  that $M G$ is measurable.
  Further, note that if $\xi \!\in\! n + [-\frac{1}{2,} \frac{1}{2}]^d$,
  then $n \!\in\! \xi + [-\frac{1}{2,} \frac{1}{2}]^d$ and thus $M G (\xi) \geq |G(n)|$.
  Therefore, if $M G \in L^1$, then
  \begin{equation}
    \begin{split}
      \sum_{n \in \Z^d}
        |G(n)|
      & = \sum_{n \in \Z^d}
            \int_{n + [-1/2, 1/2]^d}
              |G (n)|
            \, d \xi \\
      & \leq \sum_{n \in \Z^d}
               \int_{n + [-1/2, 1/2]^d}
                 M G (\xi)
               \, d \xi \\
      & = \int_{\R^d}
            M G (\xi)
          \, d \xi
        = \| M G \|_{L^1}
      .
    \end{split}
    \label{eq:SamplingUsingMaximalFunction}
  \end{equation}

  Finally, we note that $\widehat{\varphi} \in \Schwartz(\R^d)$, so that there exists
  a certain constant $\kappa_1 = \kappa_1 (\varphi, \sigma) = \kappa_1(d,\sigma) \geq 1$ satisfying
  $|\widehat{\varphi}(\xi)| \leq \kappa_1 \cdot (1 + |\xi|)^{-(d+1+\sigma)}$ for all $\xi \in \R^d$.
  This implies for any $\eta \in \R^d$ and $v \in [-1/2, 1/2]^d$ that
  \begin{equation}
    w(\xi + v) \cdot |\widehat{\varphi}(\xi + v)|
    \leq \kappa_1 \cdot (1 + |\xi + v|)^{-(d+1)}
    \leq \kappa_2 \cdot (1 + |\xi|)^{-(d+1)}
    \label{eq:CutoffFunctionFourierDecay}
  \end{equation}
  for a suitable constant $\kappa_2 = \kappa_2 (d,\sigma) \geq 1$.
  Here, the last step used that
  \(
    1 + |\xi|
    \leq 1 + |\xi + v| + |-v|
    \leq 1 + \sqrt{d} + |\xi + v|
    \leq (1 + \sqrt{d}) \cdot (1 + |\xi + v|)
  \)
  and thus $(1 + |\xi + v|)^{-(d+1)} \leq (1 + \sqrt{d})^{d+1} \cdot (1 + |\xi|)^{-(d+1)}$.

  \medskip{}

  \noindent
  \textbf{Step~2 (A Fourier expansion for $\varphi \cdot f$):}
  Let $f \in B^\sigma(\Omega)$ with $\| f \|_{B^\sigma(\Omega)} \leq 1$ be arbitrary.
  By definition of this norm, this means that there exists $F \in L_w^1(\R^d)$
  with $\| F \|_{L_w^1} \leq 2$ and such that
  $f(x) = \int_{\R^d} e^{2 \pi i \langle x, \xi \rangle} F(\xi)\, d \xi$
  for all $x \in \Omega$.
  We can thus extend $f$ to a function defined on all of $\R^d$ by setting
  $f := \Fourier^{-1} F$.
  Note that $f$ is then continuous and bounded with $\| f \|_{L^\infty} \leq \| F \|_{L^1} \leq 2$.

  Now, define
  \(
    H
    := F \ast \widehat{\varphi}
    \in L_w^1 \ast L^\infty
    \subset L^1 \ast L^\infty
    \subset C_b
    ,
  \)
  where $C_b$ is the space of all bounded continuous functions on $\R^d$.
  Using \Cref{eq:WeightSubmultiplicative,eq:CutoffFunctionFourierDecay},
  we then see for arbitrary $\xi \in \R^d$ and $v \in [-1/2, 1/2]^d$ that
  \begin{align*}
    |H_w (\xi + v)|
    & \leq \int_{\R^d}
             w(\xi + v) \cdot |F(\eta)| \cdot |\widehat{\varphi}(\xi + v - \eta)|
           \, d \eta \\
    & \leq \int_{\R^d}
             w(\eta) \cdot w(\xi + v - \eta) \cdot |F(\eta)| \cdot |\widehat{\varphi}(\xi + v - \eta)|
           \, d \eta \\
    & \leq \kappa_2
           \int_{\R^d}
             |F_w (\eta)| \cdot (1 + |\xi - \eta|)^{-(d+1)}
           \, d \eta ,
  \end{align*}
  which means that $M H_w (\xi) \leq \kappa_2 \cdot |F_w| \ast (1 + |\bullet|)^{-(d+1)}$.
  Using the well-known estimate $\| f \ast g \|_{L^1} \leq \| f \|_{L^1} \cdot \| g \|_{L^1}$
  for convolutions, we thus see
  \[
    \| M H_w \|_{L^1}
    \leq \kappa_2 \cdot \| F_w \|_{L^1} \cdot \| (1 + |\bullet|)^{-(d+1)} \|_{L^1}
    \leq \kappa_3
  \]
  for a suitable constant $\kappa_3 = \kappa_3(d,\sigma) \geq 1$.
  In view of \Cref{eq:SamplingUsingMaximalFunction}, this implies
  \[
    \sum_{n \in \Z^d}
      (1 + |n|)^\sigma |H(n)|
    = \sum_{n \in \Z^d}
        |H_w(n)|
    \leq \| M H_w \|_{L^1}
    \leq \kappa_3 .
  \]
  Now, finally note because of $F, \widehat{\varphi} \in L^1(\R^d)$ and by the convolution theorem
  that
  \[
    \Fourier^{-1} H
    = \Fourier^{-1} \bigl[ \, F \ast \widehat{\varphi} \, \bigr]
    = (\Fourier^{-1} F) \cdot \varphi
    = f \cdot \varphi
    ,
  \]
  so that the Fourier coefficients of $f \cdot \varphi$
  (considered as a function on $\Omega = [-\frac{1}{2}, \frac{1}{2}]^d$) are given by
  \begin{align*}
    c_n
    & := \widehat{f \cdot \varphi} (n)
    = \int_{[-1/2, 1/2]^d}
        f(x) \, \varphi(x) \, e^{-2 \pi i \langle n, x \rangle}
      \, d x \\
    & = \int_{\R^d}
          f(x) \, \varphi(x) \, e^{- 2 \pi i \langle n, x \rangle}
        \, d x \\
    & = \Fourier [f \cdot \varphi] (n)
      = \Fourier \Fourier^{-1} H (n)
      = H (n)
    .
  \end{align*}
  Here, the last step is justified by Fourier inversion, since $H \in L^1$ is continuous
  and $\Fourier^{-1} H = f \cdot \varphi \in L^1$ as well.

  Overall, since $(e^{- 2 \pi i \langle n, \bullet \rangle})_{n \in \Z^d}$ is an orthonormal basis
  for $L^2(\Omega)$, we see that
  \begin{equation}
    f \cdot \varphi
    = \sum_{n \in \Z^d}
        c_n \, e^{2 \pi i \langle n ,\bullet \rangle}
    \qquad \text{with} \qquad
    \sum_{n \in \Z^d}
      (1 + |n|)^\sigma \cdot |c_n|
    \leq \kappa_3
    .
    \label{eq:LocalizedFFourierExpansion}
  \end{equation}
  The convergence of the series is a priori in $L^2(\Omega)$, but since $f \cdot \varphi$
  is continuous and the Fourier series converges absolutely (and hence uniformly),
  the equality in fact holds pointwise.

  \medskip{}

  \noindent
  \textbf{Step~3 (Approximating $\varphi \cdot f$ by a sparse trigonometric polynomial):}
  For a suitable constant $\kappa_4 = \kappa_4 (d,\sigma) \geq 1$ to be specified below, define
  \begin{equation}
    \lambda := \frac{1}{2} + \frac{\sigma}{d},
    \quad
    N := \big\lfloor \big( 2 \kappa_3 \big/ \eps \big)^{1/\sigma} \big\rfloor,
    \quad \text{and} \quad
    s := \bigg\lceil \left(\frac{3^{d+1} \kappa_3 \kappa_4}{\eps}\right)^{\!\! 1/\lambda} \bigg\rceil
    .
    \label{eq:ParameterChoices}
  \end{equation}
  Furthermore, set
  \begin{equation}
    I := \{ k \in \Z^d \colon \| k \|_{\ell^\infty} \leq 2N \}
    .
    \label{eq:IndexSet}
  \end{equation}

  For brevity, define $f_0 := f \cdot \varphi$.
  \Cref{eq:LocalizedFFourierExpansion} shows for $c_k := \widehat{f_0}(k)$
  that $f_0 = \sum_{k \in \Z^d} c_k \, e_k$ with $\sum_{k \in \Z^d} (1 + |k|)^\sigma |c_k| \leq \kappa_3$.
  Therefore, setting $g := \sum_{\| k \|_{\ell^\infty} \leq N} c_k \, e_k$, we see
  \begin{equation}
    \begin{split}
      \| f_0 - g \|_{L^\infty}
      & \leq \sum_{\| k \|_{\ell^\infty} > N} |c_k| \\
      & \leq (1 + N)^{-\sigma}
             \sum_{\| k \|_{\ell^\infty} \geq N}
               (1 + \| k \|_{\ell^\infty})^\sigma |c_k| \\
      & \leq (1 + N)^{-\sigma} \sum_{\| k \|_{\ell^\infty} \geq N} (1 + | k |)^\sigma |c_k| \\
      & \leq (1 + N)^{-\sigma} \, \kappa_3
        \leq \frac{\eps}{2}
        ,
    \end{split}
    \label{eq:CutoffEstimate}
  \end{equation}
  where the last step is justified by our specific choice of $N$.
  Now, \cite[Theorem~6.1]{DeVoreNonLinearApproximationByTrigonometricSum} shows that there
  exists a constant $\kappa_4 = \kappa_4(d,\sigma) \geq 1$ (already mentioned above) and a trigonometric
  polynomial $h_0$ with at most $s$ nonzero terms (i.e., a linear-combination of at most $s$
  elements of the family $(e_k)_{k \in \Z}$) such that
  $\| \frac{g}{\kappa_3} - h_0 \|_{L^\infty} \leq \kappa_4 \cdot s^{-\lambda}$,
  and thus
  \[
    \| g - h \|_{L^\infty}
    \leq \kappa_3 \kappa_4 \cdot s^{-\lambda}
    \leq \frac{\eps}{3^{d+1}}
  \]
  for $h := \kappa_3 \, h_0$, which is also a trigonometric polynomial with at most $s$ non-zero terms.

  Next, for $k \in \N$, let
  \[
    D_k(x)
    := \sum_{\ell=-k}^k
         e^{2 \pi i \ell x},
    \qquad
    F_k
    := \frac{1}{k}
       \sum_{\ell=0}^{k-1}
         D_\ell,
    \qquad \text{and} \qquad
    V_k
    := (1 + e_k + e_{-k}) \cdot F_k
  \]
  denote the (one-dimensional) Dirichlet kernel, Fejér kernel, and de la Vallée Poussin kernel,
  respectively.
  We extend these one-dimensional kernels to the $d$-dimensional kernels $D_k^d, F_k^d, V_k^d$
  where $D_k^d (x_1,\dots,x_d) := D_k(x_1) \cdots D_k(x_d)$ and similarly for the other kernels.
  It is well-known (see e.g.\ \cite[Pages~9 and 15]{MuscaluSchlagHA})
  that $\| F_k \|_{L^1([-1/2,1/2])} = 1$ and that $\widehat{V_k}(j) = 1$ for $|j| \leq k$.
  Therefore, $\| V_k \|_{L^1([-1/2,1/2])} \leq 3$ and $\| V_k^d \|_{L^1(\Omega)} \leq 3^d$
  and furthermore $\widehat{V_k^d}(j) = 1$ if $\| j \|_{\ell^\infty} \leq k$.

  This easily implies that
  \(
    \widehat{g}
    = \widehat{g} \cdot \widehat{\vphantom{V}\smash{V_N^d}}
    = \widehat{g \ast \vphantom{V}\smash{V_N^d}}
  \)
  and hence $g = g \ast V_N^d$.
  Furthermore, $V_N^d \in \linspan \{ e_\ell \colon \ell \in I \}$,
  and this implies (e.g.\ by considering the Fourier coefficients)
  that $f^\ast := h \ast V_N^d \in \linspan \{ e_\ell \colon \ell \in I \}$
  is a trigonometric polynomial with at most $s$ non-zero terms.
  Moreover,
  \[
    \| g - f^\ast \|_{L^\infty}
    = \| (g - h) \ast V_N^d \|_{L^\infty}
    \leq \| V_N^d \|_{L^1} \cdot \| g - h \|_{L^\infty}
    \leq \frac{\eps}{3}
    .
  \]
  In view of \Cref{eq:CutoffEstimate} and because of $f_0 = f \cdot \varphi$, this implies
  \begin{equation}
    \begin{split}
      & \| f \cdot \varphi - f^\ast \|_{L^\infty}
      \leq \eps \\
      \text{where} \,\,
      & f^\ast \!\in\! \linspan \{ e_\ell \colon \ell \!\in\! I \}
        \text{ is a trig.\ pol.\ with at most $s$ nonzero terms}
      .
    \end{split}
    \label{eq:ApproximationBySparsePolynomial}
  \end{equation}

  \medskip{}

  \noindent
  \textbf{Step~4 (Applying bounds from compressive sensing):}
  By \cite[Corollary~12.34 (b)]{RauhutFoucartIntroToCS}, there exist universal constants
  $\kappa_5,\kappa_6 \geq 1$ such that if
  \begin{equation}
    m \geq \kappa_5 \cdot s \cdot \ln^4 (|I|)
    ,
    \label{eq:SamplePointCondition}
  \end{equation}
  then there exist sampling points $x_1,\dots,x_m \in \Omega$ such that the following holds:
  If $f^\ast = \sum_{\ell \in I} \zeta_\ell^\ast \, e_\ell \in \linspan \{ e_\ell \colon \ell \in I \}$
  is a trigonometric polynomial with at most $s$ non-zero terms (i.e., $\zeta_\ell^\ast \neq 0$ for
  at most $s$ different choice of the index $\ell$), if $\eta > 0$,
  and if $y = \bigl(f^\ast (x_i)\bigr)_{i=1}^m + e$ with $\| e \|_{\ell^2} \leq \eta \sqrt{m}$,
  then any solution $\zeta^{\sharp} = \zeta^{\sharp} (y;\eta) \in \CC^{I}$ of the optimization problem
  \begin{equation}
    \operatornamewithlimits{minimize}_{\zeta \in \CC^I} \, \| \zeta \|_{\ell^1}
    \quad \text{subject to} \quad
    \bigg\| y - \bigg(\sum_{\ell \in I} \zeta_\ell \, e_\ell (x_i) \bigg)_{i=1}^m \bigg\|_{\ell^2}
    \leq \eta \sqrt{m}
    \label{eq:CSOptimizationProblem}
  \end{equation}
  satisfies
  \[
    \| \zeta^\ast - \zeta^\sharp \|_{\ell^2}
    \leq \kappa_6 \cdot \eta .
  \]
  Because of $\eps < 1$, it is easy to see that one can ensure
  by proper choice of the constant $\kappa = \kappa(d,\sigma)$
  in our assumption regarding $m$ (\Cref{eq:AssumptionNumberOfSamples})
  that \Cref{eq:SamplePointCondition} is satisfied.

  Now, we define the map $T : \R^m \to L^2(\Omega)$ as follows:
  Given any vector $y \in \R^m$, we set $z(y) := \bigl(y_i \cdot \varphi(x_i)\bigr)_{i=1}^m$
  and we choose
  \[
    T (y)
    := \sum_{\ell \in I}
         \zeta^{\sharp}_\ell \bigl(z(y); \eps\bigr) \, e_\ell \in L^2(\Omega)
    ,
  \]
  where $\zeta^{\sharp} \bigl(z(y);\eps\bigr) \in \CC^I$ is any (fixed) solution
  of the optimization problem in \Cref{eq:CSOptimizationProblem}
  (with $\eta = \eps$ and $z(y)$ instead of $y$).
  It remains to prove that this function satisfies the required estimate.
  To this end, let $f \in B^\sigma(\Omega)$ with $\| f \|_{B^\sigma} \leq 1$ be arbitrary,
  set $f_0 := f \cdot \varphi$ for brevity, and choose
  $f^\ast = \sum_{\ell \in I} \zeta_\ell^\ast \, e_\ell$
  as in Step~3 (see \Cref{eq:ApproximationBySparsePolynomial}).
  Furthermore, let $y = \bigl(f(x_i)\bigr)_{i=1}^m$, noting that $z(y) = \bigl(f_0(x_i)\bigr)_{i=1}^m$.
  Finally, let $e := \big( f_0(x_i) - f^\ast(x_i) \big)_{i=1}^m$, noting that
  $\| e \|_{\ell^\infty} \leq \eps$ and hence $\| e \|_{\ell^2} \leq \eps \sqrt{m}$
  and that $z(y) = \bigl(f^\ast (x_i)\bigr)_{i=1}^m + e$.
  Now, the error estimate from above shows that
  $\| \zeta^\ast - \zeta^{\sharp}(z(y);\eps) \|_{\ell^2} \leq \kappa_6 \cdot \eps$
  and thus
  \begin{align*}
    \| f \cdot \varphi - T(f(x_1),\dots,f(x_m)) \|_{L^2}
    & \leq \| f_0 - f^\ast \|_{L^2} + \| f^\ast - T(y) \|_{L^2} \\
    & \leq \eps + \| \zeta^\ast - \zeta^{\sharp}(z(y);\eps) \|_{\ell^2} \\
    & \leq \eps + \kappa_6 \cdot \eps
      =    C \cdot \eps
  \end{align*}
  for $C := 1 + \kappa_6$.
  Note that $C$ is an absolute constant, independent of $d,\sigma$.
  Since $\varphi \equiv 1$ on $\Omega_0$, this proves the desired estimate.
\end{proof}

Our goal is to ``amplify'' the above lemma to yield
reconstruction on all of $\Omega$, not just on $\Omega_0$.
As a technical ingredient for the proof, we will need the following general result regarding
dilating and translating Barron functions.

\begin{lemma}\label{lem:TranslationDilation}
  Let $s \in (0,\infty)$ and $\Omega, \Theta \subset \R^d$ be arbitrary,
  and let $t_0 \in \R^d$ and $a = (a_1,\dots,a_d) \in ([-1,1] \setminus \{ 0 \})^d$ such that
  \[
    \varphi : \quad
    \R^d \to \R^d, \quad
    t \mapsto D_a \, t + t_0
  \]
  satisfies $\varphi(\Theta) \subset \Omega$.

  Then, it holds for any $f \in B^\sigma(\Omega)$ that $f \circ \varphi \in B^\sigma(\Theta)$ with
  \[
    \| f \circ \varphi \|_{B^\sigma(\Theta)}
    \leq \| f \|_{B^\sigma(\Omega)}
    .
  \]
\end{lemma}

\begin{proof}
  Let $f \in B^\sigma(\Omega)$ be arbitrary.
  For any $\delta > 0$, there then exists a measurable $F : \R^d \to \CC$ satisfying
  $f(x) = \int_{\R^d} e^{2 \pi i \langle x,\xi \rangle} \, F(\xi) \, d \xi$ for all $x \in \Omega$,
  and $\int_{\R^d} (1 + |\xi|)^\sigma \cdot |F(\xi)|\, d \xi \leq \delta + \| f \|_{B^\sigma(\Omega)}$.

  Now, define $G : \R^d \to \CC$ by
  $G(\eta) := e^{2 \pi i \langle D_a^{-1} t_0, \eta \rangle} \cdot |\det D_a|^{-1} \cdot F(D_a^{-1} \eta)$.
  On the one hand, we then see for arbitrary $t \in \Theta$
  (for which then $x := D_a \, t + t_0 \in \Omega$ by assumption) that
  \begin{align*}
    (f \circ \varphi) (t)
    & = f(D_a \, t + t_0)
      = \int_{\R^d}
          e^{2 \pi i \langle D_a \, t + t_0, \xi \rangle} \, F(\xi)
        \, d \xi \\
    & = \int_{\R^d}
          |\det D_a|
          \cdot e^{2 \pi i \langle t + D_a^{-1} t_0 , D_a \xi \rangle}
          \cdot |\det D_a|^{-1} \cdot F(D_a^{-1} D_a \, \xi)
        \, d \xi \\
    & = \int_{\R^d}
          e^{2 \pi i \langle t + D_a^{-1} t_0, \eta \rangle}
          \cdot |\det D_a|^{-1} \cdot F(D_a^{-1} \eta)
        \, d \eta \\
    & = \int_{\R^d}
          e^{2 \pi i \langle t, \eta \rangle} \cdot G(\eta)
        \, d \eta
    .
  \end{align*}
  On the other hand, we see because of $|a_j| \leq 1$ for all $j \in \FirstN{d}$
  that $|D_a \xi| \leq |\xi|$ for all $\xi \in \R^d$, and thus
  \begin{align*}
    \int_{\R^d}
      (1 + |\eta|)^\sigma \cdot |G(\eta)|
    \, d \eta
    & = \int_{\R^d}
          (1 + |D_a \, D_a^{-1} \eta|)^\sigma \cdot |\det D_a|^{-1} \cdot |F(D_a^{-1} \eta)|
        \, d \eta \\
    & = \int_{\R^d}
          (1 + |D_a \, \xi|)^\sigma \cdot |F(\xi)|
        \, d \xi \\
    & \leq \int_{\R^d}
             (1 + |\xi|)^\sigma \cdot |F(\xi)|
           \, d \xi
      \leq \delta + \| f \|_{B^\sigma(\Omega)}
    .
  \end{align*}
  This shows that $f \circ \varphi \in B^\sigma(\Theta)$ with
  $\| f \circ \varphi \|_{B^\sigma(\Theta)} \leq \delta + \| f \|_{B^\sigma(\Omega)}$.
  Since this holds for every $\delta > 0$, we are done.
\end{proof}

As announced above, we now show how \Cref{lem:ReconstructionOnSmallerSet} from above
can be ``amplified'' to yield reconstruction on all of $\Omega$, not just on $\Omega_0$.
The proof idea is roughly as follows:
We decompose the unit cube $[-\frac{1}{2}, \frac{1}{2}]^d \strut$ into countably many disjoint
$d$-dimensional rectangles $Q_{\n}^\theta$, such that each
$\strut Q_{\n}^\theta \subset P_{\n}^\theta \subset [-\frac{1}{2}, \frac{1}{2}]^d$, where
the rectangles $Q_{\n}^\theta, P_{\n}^{\theta}$ can be transformed by an affine-linear transformation
$\Phi_{\n}^\theta$ into the ``standard setting'' of the cubes
$[-\frac{1}{4}, \frac{1}{4}]^d$ and $[-\frac{1}{2}, \frac{1}{2}]^d$
considered in \Cref{lem:ReconstructionOnSmallerSet}.
This necessitates that the size of each $Q_{\n}^\theta$ is comparable
to the distance of $Q_{\n}^\theta$ to the boundary of $[-\frac{1}{2}, \frac{1}{2}]^d$.
After performing a suitable coordinate change, one can then use \Cref{lem:ReconstructionOnSmallerSet}
to obtain a reconstruction of $f|_{Q_{\n}^\theta}$ for sufficiently many $\n \in I \subset \N_0^d$
such that the measure of the remaining set
$[-\frac{1}{2}, \frac{1}{2}]^d \setminus \bigcup_{\n \in I} Q_{\n}^{\theta}$
is sufficiently small.
By ``gluing together'' all the reconstructions of all the $f_{Q_{\n}^\theta}$ for $\n \in I$,
one then obtains the overall reconstruction.

\begin{theorem}\label{thm:UpperBoundL2}
  Let $d \in \N$ and $\sigma \in (0,\infty)$ be arbitrary and set
  \[
    \lambda
    := \frac{1}{2} + \frac{\sigma}{d}
    .
  \]
  Then there exist $\kappa = \kappa(d,\sigma) \geq 1$ and $C = C(d) \geq 1$
  with the following property:
  If $\eps \in (0,1)$ and $m \in \N$ are such that
  \[
    m \geq \kappa \cdot \eps^{-1 / \lambda} \cdot \ln^4 (e + \eps^{-1}) ,
  \]
  then there exist sampling points $x_1,\dots,x_m \in \Omega := [-\frac{1}{2}, \frac{1}{2}]^d$
  and a reconstruction map $T : \R^m \to L^2(\Omega)$ satisfying
  \[
    \forall \, f \in B^\sigma (\Omega) \text{ with } \| f \|_{B^\sigma(\Omega)} \leq 1:
    \quad \| f - T(f(x_1),\dots,f(x_m)) \|_{L^2(\Omega)} \leq C \cdot \eps
    .
  \]
\end{theorem}

\begin{proof}
  \textbf{Step~1 (Constructing a suitable partition of $\Omega$ and associated affine-linear maps):}
  For $n \in \N_0$, define
  \[
    a_n := \frac{1}{4} \sum_{j=0}^{n-1} 2^{-j}
    \quad \text{and} \quad
    I_n^+ := [a_n, a_{n+1}),
    \quad \text{as well as} \quad
    I_n^- := - I_n^+
    .
  \]
  For later use, we observe that $a_{n+1} - a_n = \frac{1}{4} 2^{-n} = 2^{-(n+2)} \strut$
  and therefore $a_{n+1} = a_n + 2^{-(n+2)}$.
  Furthermore, we note that $\strut [0,\frac{1}{2}) = \biguplus_{n=0}^\infty I_n^+$,
  where the union is disjoint, and hence
  $(-\frac{1}{2}, \frac{1}{2}) = \bigcup_{n=0}^\infty I_n^+ \cup I_n^-$,
  where the union is disjoint up to null-sets.

  Next, define $\varphi_n^+, \psi_n^+, \varphi_n^-, \psi_n^- : \R \to \R$ by
  \[
    \varphi_n^+(x) := 2^{n+1} \cdot \left(x - \frac{a_{n+1} + a_n}{2}\right),
    \qquad
    \psi_n^+ (t) := 2^{-(n+1)} t + \frac{a_{n+1} + a_n}{2}
    ,
  \]
  as well as $\varphi_n^- (x) := \varphi_n^+ (-x)$ and $\psi_n^- (t) := -\psi_n^+ (t)$.
  It is straightforward to check that $\varphi_n^+(\psi_n^+ (t)) = t$
  and $\psi_n^+ (\varphi_n^+(x)) = x$ for all $t,x \in \R$.
  We now claim that
  \begin{equation}
    \psi_n^+ ([-\tfrac{1}{2}, \tfrac{1}{2}]) \subset [-\tfrac{1}{2}, \tfrac{1}{2}]
    \qquad \text{and} \qquad
    \varphi_n^+ (I_n^+) = [-\tfrac{1}{4}, \tfrac{1}{4} )
    .
    \label{eq:AffineMapsProperties1}
  \end{equation}
  To see that the first of these properties is true, note
  that if $t \in [-\frac{1}{2}, \frac{1}{2}]$, then we have because of
  $\frac{1}{2} (a_{n+1} + a_n) = \frac{1}{2} (2 a_n + 2^{-(n+2)}) = a_n + 2^{-(n+3)}$ that
  \[
    \psi_n^+ (t)
    \geq - 2^{-(n+2)} + a_n + 2^{-(n+3)}
    \geq - 2^{-(n+3)}
    \geq - 2^{-3}
    \geq -\frac{1}{4}
  \]
  and
  \[
    \psi_n^+ (t)
    \leq 2^{-(n+2)} + a_n + 2^{-(n+3)}
    =    \frac{1}{4} \cdot \bigg(2^{-n} + \sum_{j=0}^{n-1} 2^{-j} + 2^{-(n+1)} \bigg)
    <    \frac{1}{4} \sum_{j=0}^{\infty} 2^{-j}
    =    \frac{1}{2}
    .
  \]
  To see that the second property in \Cref{eq:AffineMapsProperties1} is true, note for
  any half-open interval $I = [a,b)$ with $a < b$ that $x \in I$ if and only if
  $- \frac{b-a}{2} \leq x - \frac{a+b}{2} < \frac{b-a}{2}$.
  Therefore, we have
  \begin{align*}
    x \in I_n^{+} \quad
    & \Longleftrightarrow \quad
    - \frac{a_{n+1} - a_n}{2} \leq x - \frac{a_n + a_{n+1}}{2} < \frac{a_{n+1} - a_n}{2} \\
    & \Longleftrightarrow \quad
    - 2^{-(n+3)} \leq x - \frac{a_n + a_{n+1}}{2} < 2^{-(n+3)} \\
    & \Longleftrightarrow \quad
    - \frac{1}{4} \leq 2^{n+1} \cdot \left(x - \frac{a_{n+1} + a_n}{2}\right) < \frac{1}{4} \\[0.2cm]
    & \Longleftrightarrow \quad
    \varphi_n^+ (x) \in [-\tfrac{1}{4}, \tfrac{1}{4} )
    .
  \end{align*}
  Finally, note that \Cref{eq:AffineMapsProperties1} also implies that
  \begin{equation}
    \psi_n^{-} ([-\tfrac{1}{2}, \tfrac{1}{2}]) \subset [-\tfrac{1}{2}, \tfrac{1}{2}]
    \qquad \text{and} \qquad
    \varphi_n^{-}(I_n^-) = [-\tfrac{1}{4}, \tfrac{1}{4} )
    .
    \label{eq:AffineMapsProperties2}
  \end{equation}

  Now, we ``lift'' the constructed partition of $(-\frac{1}{2}, \frac{1}{2})$ to a partition
  of $\Omega$ (up to a null-set).
  To this end, define $\Theta := \{ +, - \}^d$.
  Next, for $\n = (n_1,\dots,n_d) \in \N_0^d$ and $\theta = (\theta_1,\dots,\theta_d) \in \Theta$,
  define
  \[
    Q_{\n}^\theta
    := \prod_{j=1}^d
         I_{n_j}^{\theta_j}
    \quad \text{and} \quad
    \Phi_{\n}^\theta := \varphi_{n_1}^{\theta_1} \times \cdots \times \varphi_{n_d}^{\theta_d}
    \quad \text{as well as} \quad
    \Psi_{\n}^{\theta} := \psi_{n_1}^{\theta_1} \times \cdots \times \psi_{n_d}^{\theta_d}
    ,
  \]
  where we use the notation
  \[
    \gamma_1 \times \cdots \times \gamma_d : \quad
    \R^d \to \R^d, \quad
    x \mapsto \big( \gamma_1(x_1), \dots, \gamma_d(x_d) \big)
  \]
  for functions $\gamma_1,\dots,\gamma_d : \R \to \R$.
  From the preceding observations, we easily obtain the following properties:
  \begin{enumerate}
    \item $\Omega = \bigcup_{\n \in \N_0^d, \theta \in \Theta} Q_{\n}^{\theta}$,
          where the union is disjoint up to null-sets;

    \item $\Phi_{\n}^\theta, \Psi_{\n}^\theta : \R^d \to \R^d$ are bijective and inverse to each other;

    \item $\Phi_{\n}^\theta (Q_{\n}^\theta) = [-\frac{1}{4}, \frac{1}{4})^d$;

    \item $\Psi_{\n}^\theta (\Omega) \subset \Omega$.
  \end{enumerate}

  \medskip{}

  \noindent
  \textbf{Step~2 (Constructing the sampling points and the reconstruction map):}
  Let $\kappa_0 = \kappa_0(d,\sigma) \geq 1$ and $C_1 \geq 1$
  be the constants from \Cref{lem:ReconstructionOnSmallerSet}.
  Let $\eps \in (0,1)$ be arbitrary and set
  \[
    I := \big\{
           \n \in \N_0^d
           \colon
           \eps \cdot 2^{\| \n \|_{\ell^1} / 4} < 1
         \big\}
    .
  \]
  For $\n \in I$, define
  \[
    m(\n)
    := \bigg\lceil
         \kappa_0
         \cdot \Bigl( \eps \cdot 2^{\| \n \|_{\ell^1} / 4} \Bigr)^{-1/\lambda}
         \cdot \ln^4 \bigg( e + \frac{1}{\eps \cdot 2^{\| \n \|_{\ell^1} / 4}} \bigg)
       \bigg\rceil
    .
  \]
  Then, for each $\n \in I$, \Cref{lem:ReconstructionOnSmallerSet} shows that there exist
  a (\emph{not} necessarily linear) operator $T_{\n} : \R^{m(\n)} \to L^2(\Omega)$
  and points $z_1^{(\n)},\dots,z_{m(\n)}^{(\n)} \in \Omega$ satisfying
  \begin{equation}
    \big\| f - T_{\n} \bigl( f(z_1^{(\n)}), \dots, f(z_{m(\n)}^{(\n)}) \bigr) \big\|_{L^2(\Omega_0)}
    \leq C_1 \cdot \eps \cdot 2^{\| \n \|_{\ell^1} / 4}
    \label{eq:ReconstructionBoundVariablePointNumber}
  \end{equation}
  for all $f \in B^{\sigma}(\Omega)$ with $\| f \|_{B^\sigma (\Omega)} \leq 1$;
  here $\Omega_0 = [-\frac{1}{4}, \frac{1}{4}]^d$.
  Now, given $\n \in I$ and $\theta \in \Theta$, define
  \[
    z_{i}^{\n,\theta}
    := \Psi_{\n}^{\theta} (z_i^{(\n)})
    \quad \text{for} \quad
    i \in \FirstN{m(\n)}
    ,
  \]
  noting that $z_i^{\n,\theta} \in \Omega$, since $\Psi_{\n}^\theta (\Omega) \subset \Omega$.

  Next, note that if $\n \in I$, then
  \(
    \kappa_0
    \cdot \Bigl( \eps \cdot 2^{\| \n \|_{\ell^1} / 4} \Bigr)^{-1/\lambda}
    \cdot \ln^4 \bigg( e + \frac{1}{\eps \cdot 2^{\| \n \|_{\ell^1} / 4}} \bigg)
    \geq 1
  \)
  and thus
  \begin{align*}
    m(\n)
    & \leq 1 + \kappa_0
               \cdot \Bigl( \eps \cdot 2^{\| \n \|_{\ell^1} / 4} \Bigr)^{-1/\lambda}
               \cdot \ln^4 \bigg( e + \frac{1}{\eps \cdot 2^{\| \n \|_{\ell^1} / 4}} \bigg) \\
    & \leq 2 \, \kappa_0
           \cdot \Bigl( \eps \cdot 2^{\| \n \|_{\ell^1} / 4} \Bigr)^{-1/\lambda}
           \cdot \ln^4 \bigg( e + \frac{1}{\eps \cdot 2^{\| \n \|_{\ell^1} / 4}} \bigg) \\[0.2cm]
    & \leq 2 \kappa_0
           \cdot \eps^{-1/\lambda}
           \cdot \ln^4 (e + \eps^{-1})
           \cdot 2^{- \| \n \|_{\ell^1} / (4 \lambda)}
    .
  \end{align*}
  Therefore, we see that if the constant $\kappa = \kappa(d,\sigma) \geq 1$
  from the statement of the theorem is chosen sufficiently large, then
  \begin{align*}
    \sum_{\n \in I, \theta \in \Theta}
      m(\n)
    & \leq 2^{d+1} \kappa_0
           \cdot \eps^{-1/\lambda}
           \cdot \ln^4 (e + \eps^{-1})
           \cdot \sum_{\n \in \N_0^d}
                   2^{- \| \n \|_{\ell^1} / (4 \lambda)} \\
    & \leq \kappa
           \cdot \eps^{-1/\lambda}
           \cdot \ln^4 (e + \eps^{-1})
    \leq m .
  \end{align*}

  Therefore, we can choose $x_1,\dots,x_m \in \Omega$ such that
  \[
    \bigcup_{\n \in I, \theta \in \Theta}
      \{ z_1^{\n,\theta}, \dots, z_{m(\n)}^{\n,\theta} \}
    \subset \{ x_1,\dots,x_m \}
    .
  \]
  As a consequence of this choice, for each $\n \in I$ and $\theta \in \Theta$ we can choose a map
  $\pi_{\n}^\theta : \R^m \to \R^{m(\n)}$ (a projection onto some of the coordinates) such that
  \[
    \pi_{\n}^\theta \Big( \big( f(x_i) \big)_{i=1}^m \Big)
    = \big( f(z_\ell^{\n,\theta}) \big)_{\ell=1}^{m(\n)}
    = \big( (f \circ \Psi_{\n}^\theta) (z_\ell^{(\n)}) \big)_{\ell=1}^{m(\n)}
    .
  \]
  Now, we define the reconstruction operator $T : \R^m \to L^2(\Omega)$ as follows:
  Given $y \in \R^m$, we set
  \[
    F_{\n,\theta}^{(y)}
    := T_{\n} (\pi_{\n}^\theta (y))
    \in L^2(\Omega)
  \]
  and
  \[
    T(y)
    := \sum_{\n \in I, \theta \in \Theta}
       \Big[
         \indicator_{Q_{\n}^\theta}
         \cdot \Big(
                 F_{\n,\theta}^{(y)}
                 \circ \Phi_{\n}^{\theta}
               \Big)
       \Big]
    \in L^2(\Omega)
    .
  \]

  \medskip{}

  \noindent
  \textbf{Step~3 (Completing the proof):}
  Let $f \in B^{\sigma}(\Omega)$ with $\| f \|_{B^\sigma (\Omega)} \leq 1$ be arbitrary
  and set
  \[
    y
    := \big( f(x_i) \big)_{i=1}^m
    \in \R^m.
  \]
  For each $\n \in I$ and $\theta \in \Theta$, let
  \[
    g_{\n}^\theta
    := f \circ \Psi_{\n}^\theta,
  \]
  considered as a function on $\Omega$; this is well-defined,
  since $\Psi_{\n}^\theta (\Omega) \subset \Omega$.
  Furthermore, since $\Psi_{\n}^\theta (\Omega) \subset \Omega$ and since the affine-linear map
  $\Psi_{\n}^\theta$ is of the form $\Psi_{\n}^\theta (x) = D_a \, x + b$ for a certain
  $b = b(\n,\theta) \in \R^d$ and $a_j = 2^{-(n_j + 1)} \in (0,1]$ if $\theta_j \in \{ + \}$
  and $a_j = - 2^{-(n_j + 1)} [-1,0)$ if $\theta_j \in \{ - \}$, we see by \Cref{lem:TranslationDilation}
  that $g_{\n}^\theta \in B^{\sigma}(\Omega)$ with $\| g_{\n}^\theta \|_{B^\sigma(\Omega)} \leq 1$.
  Therefore, \Cref{eq:ReconstructionBoundVariablePointNumber} shows that
  \[
    h_{\n}^\theta
    := T_{\n} \big( g_{\n}^\theta (z_\ell^{(\n)}) \big)_{\ell=1}^{m(\n)}
    \in L^2(\Omega)
    \quad \text{satisfies} \quad
    \| g_{\n}^\theta - h_{\n}^\theta \|_{L^2(\Omega_0)}
    \leq C_1 \cdot \eps \cdot 2^{\| \n \|_{\ell^1} / 4}
    .
  \]

  By choice of $\pi_{\n}^\theta$, we have
  \[
    \pi_{\n}^\theta (y)
    = \pi_{\n}^\theta \Big( \big( f(x_i) \big)_{i=1}^m \Big)
    = \Big( (f \circ \Psi_{\n}^\theta) (z_\ell^{(\n)}) \Big)_{\ell=1}^{m(\n)}
    = \big( g_{\n}^\theta (z_\ell^{(\n)}) \big)_{\ell=1}^{m(\n)}
  \]
  and thus $F_{\n,\theta}^{(y)} \!=\! T_{\n} (\pi_{\n}^\theta (y)) \!=\! h_{\n}^\theta$.
  Since $|\det D \Phi_{\n}^\theta (x)| \!=\! \prod_{j=1}^d 2^{n_j + 1} \!=\! 2^d \, 2^{\| \n \|_{\ell^1}}$
  and $\Phi_{\n}^\theta (Q_{\n}^\theta) \subset \Omega_0$ and since
  $f = f \circ \Psi_{\n}^\theta \circ \Phi_{\n}^\theta = g_{\n}^\theta \circ \Phi_{\n}^\theta$,
  we see using a change of variables that
  \begin{equation}
    \begin{split}
      \big\| f - (h_{\n}^\theta \circ \Phi_{\n}^\theta) \big\|_{L^2(Q_{\n}^\theta)}^2
      & = \big\| (g_{\n}^\theta - h_{\n}^\theta) \circ \Phi_{\n}^\theta \big\|_{L^2(Q_{\n}^\theta)}^2 \\
      & = 2^{-d} \,
          2^{- \| \n \|_{\ell^1}}
          \cdot \big\| g_{\n}^\theta - h_{\n}^\theta \big\|_{L^2(\Phi_{\n}^\theta (Q_{\n}^\theta))}^2 \\
      & \leq 2^{-d} \,
             2^{- \| \n \|_{\ell^1}}
             \cdot \big\| g_{\n}^\theta - h_{\n}^\theta \big\|_{L^2(\Phi_{\n}^\theta (\Omega_0))}^2 \\
      & \leq \frac{C_1^2}{2^d} \cdot \eps^2 \cdot 2^{- \| \n \|_{\ell^1} / 2}
      .
    \end{split}
    \label{eq:ErrorEstimateOnEachRectangle}
  \end{equation}

  Now, since $\Omega = \bigcup_{\n \in \N_0^d, \theta \in \Theta} Q_{\n}^\theta$ up to a null-set,
  where the $Q_{\n}^\theta$ are disjoint up to null-sets, and since we have
  \(
    T (y)
    = F_{\n,\theta}^{(y)} \circ \Phi_{\n}^\theta
    = h_{\n}^\theta \circ \Phi_{\n}^\theta
  \)
  on $Q_{\n}^\theta$ for $\n \in I$, whereas $T(y) \equiv 0$ on $Q_{\n}^\theta$ for $\n \notin I$,
  we see
  \begin{align*}
    & \big\| f - T \big( f(x_1),\dots,f(x_m) \big) \big\|_{L^2(\Omega)}^2
      = \| f - T(y) \|_{L^2(\Omega)}^2 \\
    & = \sum_{\n \in \N_0^d, \theta \in \Theta}
        \| f - T(y) \|_{L^2(Q_{\n}^\theta)}^2 \\
    & = \sum_{\n \in I, \theta \in \Theta}
          \| f - (h_{\n}^\theta \circ \Phi_{\n}^\theta) \|_{L^2(Q_{\n}^\theta)}^2
        + \sum_{\n \in \N_0^d \setminus I, \theta \in \Theta}
            \| f \|_{L^2(Q_{\n}^\theta)}^2 \\
    & \overset{(\dagger)}{\leq}
        \sum_{\n \in I, \theta \in \Theta}
          \frac{C_1^2}{2^d} \eps^2 2^{-\| \n \|_{\ell^1} / 2}
        + \sum_{\n \in \N_0^d \setminus I, \theta \in \Theta}
            \lebesgue (Q_{\n}^\theta) \\
    & \overset{(\ddagger)}{\leq}
        C_1^2 \, \eps^2
        \sum_{\n \in I, \theta \in \Theta}
          2^{-\frac{1}{2} (n_1 + \cdots + n_d)}
        + \eps^2
          \sum_{\n \in \N_0^d \setminus I, \theta \in \Theta}
            2^{-\frac{1}{2} (n_1 + \cdots + n_d)} \\
    & \leq C_1^2 \, \eps^2 \sum_{\n \in \N_0^d} 2^{-\frac{1}{2} (n_1 + \cdots + n_d)}
      =:   C_2^2 \, \eps^2
  \end{align*}
  for a suitable $C_2 = C_2(d, C_1) = C_2(d) \geq 1$.
  Here, the step marked with $(\dagger)$ is justified by \Cref{eq:ErrorEstimateOnEachRectangle}
  and since $\| f \|_{L^\infty(\Omega)} \leq \| f \|_{B^\sigma (\Omega)} \leq 1$
  by \Cref{lem:BarronEmbedsIntoLInfty}.
  The step marked with $(\ddagger)$ is justified by noting that if $\n \notin I$,
  then $\eps \cdot 2^{\| \n \|_{\ell^1} / 4} \geq 1$ and hence $2^{-\| \n \|_{\ell^1} / 4} \leq \eps$,
  so that
  \begin{align*}
    \lebesgue (Q_{\n}^\theta)
    & = \prod_{j=1}^d
          2^{-(n_j + 2)}
      = 2^{-2 d} 2^{-\| \n \|_{\ell^1}} \\
    & \leq 2^{-d} 2^{-\| \n \|_{\ell^1} / 2} \cdot \big( 2^{- \| \n \|_{\ell^1} / 4} \big)^2
      \leq 2^{-d} \, \eps^2 \, 2^{-\| \n \|_{\ell^1} / 2}
    .
  \end{align*}
  Overall, we have thus shown as claimed that
  \[
    \big\| f - T \big( f(x_1),\dots,f(x_m) \big) \big\|_{L^2(\Omega)}
    \leq C_2 \, \eps
  \]
  for all $f \in B^\sigma (\Omega)$ with $\| f \|_{B^\sigma (\Omega)} \leq 1$.
\end{proof}

As a corollary of \Cref{thm:UpperBoundL2}, we get the following upper bound
on the $L^2$-sampling numbers for the Barron spaces.

\begin{corollary}\label{cor:L2SamplingNumberUpperBound}
  Let $d \in \N$ and $\sigma \in (0,\infty)$ be arbitrary and set
  $\Omega := [-\frac{1}{2}, \frac{1}{2}]^d$ and
  \[
    \lambda := \frac{1}{2} + \frac{\sigma}{d} 
    .
  \]
  Then there exists a constant $C = C(d,\sigma) > 0$ such that
  \[
    s_m \bigl(B^\sigma(\Omega); L^2(\Omega)\bigr)
    \leq C \cdot \Big( \ln^4 (e + m) \Big/ m \Big)^{\lambda}
    \qquad \forall \, m \in \N .
  \]
\end{corollary}

\begin{proof}
  It is easy to see that there exist constants $C_ 1 = C_1(\lambda) = C_1(d,\sigma) \geq 1$
  and $C_ 2 = C_2(C_1,\lambda) = C_2(d,\sigma) \geq 1$ satisfying for any $m \in \N$ that
  \[
    \ln (e + m^\lambda)
    \leq \ln \bigl( C_1 \cdot (e + m)^\lambda \bigr)
    \leq \ln (C_1) + \lambda \, \ln(e + m)
    \leq C_2 \cdot \ln(e + m)
    .
  \]
  Now, set $C_3 := C_2^4$ and $C := (C_3 \kappa)^{\lambda}$,
  where $\kappa = \kappa (d,\sigma) \geq 1$ is as in \Cref{thm:UpperBoundL2}.
  Note that $C \geq 1$.
  Finally, let $m \in \N$ be arbitrary and set
  \[
    \eps := \eps(m) := C \cdot \Big( \ln^4 (e + m) \Big/ m \Big)^{\lambda}
    .
  \]
  To prove the corollary, it is enough to show that
  $s_m (B^\sigma(\Omega); L^2(\Omega)) \leq C_4 \cdot \eps$
  for a constant $C_4 = C_4(d,\sigma) \geq 1$.
  To do so, we distinguish two cases.

  \smallskip{}

  \noindent
  \textbf{Case~1 ($\eps \geq 1$):} In this case, we choose $x_1,\dots,x_m \in \Omega$ arbitrary
  and choose $T : \R^m \to L^2(\Omega), y \mapsto 0$.
  Then \Cref{lem:BarronEmbedsIntoLInfty} shows for every $f \in B^\sigma(\Omega)$
  with $\| f \|_{B^\sigma(\Omega)} \leq 1$ that
  \[
    \| f - T(f(x_1),\dots,f(x_m)) \|_{L^2(\Omega)}
    \leq \| f \|_{L^\infty(\Omega)}
    \leq 1
    \leq \eps .
  \]
  This easily implies $s_m (B^\sigma(\Omega); L^2(\Omega)) \leq \eps$.

  \medskip{}

  \noindent
  \textbf{Case~2 ($\eps < 1$):} In this case, we first note because of
  $\eps \geq C \cdot m^{-\lambda} \geq m^{-\lambda}$ that
  \(
    \ln^4(e + \eps^{-1})
    \leq \ln^4(e + m^{\lambda})
    \leq C_2^4 \cdot \ln^4(e + m)
    = C_3 \cdot \ln^4(e + m)
  \)
  and thus
  \[
    \kappa \cdot \eps^{-1/\lambda} \cdot \ln^4(e + \eps^{-1})
    \leq C_3 \kappa \cdot \ln^4(e + m) \cdot C^{-1/\lambda} \cdot \frac{m}{\ln^4(e + m)}
    = m .
  \]
  Therefore, \Cref{thm:UpperBoundL2} yields a constant $C_4 = C_4(d) \geq 1$ and sampling points
  $x_1,\dots,x_m \in \Omega$ and a map $T : \R^m \to L^2(\Omega)$ satisfying for all
  $f \in B^\sigma(\Omega)$ with $\| f \|_{B^\sigma(\Omega)} \leq 1$ that
  \[
    \| f - T(f(x_1),\dots,f(x_m)) \|_{L^2(\Omega)}
    \leq C_4 \cdot \eps
    .
  \]
  As seen above, this completes the proof.
\end{proof}

\subsection{Upper bounds for \texorpdfstring{$L^\infty$}{L ᪲} sampling numbers}%
\label{sub:LInftyUpperBounds}

Our proof of the upper bounds for the $L^\infty$ sampling numbers of the Barron spaces
will simply be based on noting that $B^\sigma(\Omega) \hookrightarrow C^\sigma(\Omega)$.
Here, for $\sigma \in \N$, we define $C^\sigma(\R^d)$ as the space of all $\sigma$-times continuously
differentiable functions with all partial derivatives up to order $\sigma$ being bounded;
the associated norm is given by
\[
  \| f \|_{C^\sigma(\R^d)}
  := \max_{\alpha \in \N_0^d, |\alpha| \leq \sigma}
       \sup_{x \in \R^d}
         |\partial^\alpha f (x)|
  .
\]
Moreover, for $\sigma \in (0,\infty) \setminus \N$, we write $\sigma = k + r$
with $k \in \N_0$ and $r \in [0,1)$ and then define
$C^\sigma (\R^d) := \{ f \in C^k (\R^d) \colon |f|_{C^\sigma(\R^d)} < \infty \}$, where
\[
  |f|_{C^\sigma(\R^d)}
  := \max_{\alpha \in \N_0^d, |\alpha| = k} \,\,
       \sup_{x,y \in \R^d, x \neq y} \,\,
         \frac{|\partial^\alpha f (x) - \partial^\alpha f(y)|}{|x-y|^{r}}
  .
\]
We then define $\| f \|_{C^\sigma(\R^d)} := \| f \|_{C^k (\R^d)} + |f|_{C^\sigma(\R^d)}$.

Finally, given any non-empty set $\emptyset \neq \Omega \subset \R^d$, we define
\begin{equation}
  \begin{split}
    & C^\sigma (\Omega)
      := \big\{ f|_{\Omega} \colon f \in C^\sigma(\R^d) \big\} \\
    \text{with norm} \quad
    & \| f \|_{C^\sigma(\Omega)}
      := \inf \big\{ \| g \|_{C^\sigma(\R^d)} \colon g \in C^\sigma(\R^d) \text{ and } g|_\Omega = f \big\}
    .
  \end{split}
  \label{eq:HoelderSpacesOnDomains}
\end{equation}
The following lemma establishes the embedding $B^\sigma(\Omega) \hookrightarrow C^\sigma(\Omega)$
on which we will rely.

\begin{lemma}\label{lem:BarronIntoHoelderEmbedding}
  For every $d \!\in\! \N$ and $\sigma \in (0,\infty)$ there is a constant $C = C(\sigma) > 0$
  such that for every $\emptyset \neq \Omega \subset \R^d$ and every $f \in B^\sigma(\Omega)$,
  it holds that $f \in C^\sigma(\Omega)$ with
  $\| f \|_{C^\sigma(\Omega)} \leq C \cdot \| f \|_{B^\sigma(\Omega)}$.
\end{lemma}

\begin{proof}
  Let $f \in B^\sigma(\Omega)$ be arbitrary.
  Then, given any $\delta > 0$, there exists a measurable function $F : \R^d \to \CC$
  satisfying $\int_{\R^d} (1 + |\xi|)^\sigma |F(\xi)| \, d \xi \leq \delta + \| f \|_{B^\sigma(\Omega)}$
  and $f(x) = \int_{\R^d} F(\xi) e^{2 \pi i \langle x, \xi \rangle}\, d \xi$ for all $x \in \Omega$.
  Thus, if we define $g : \R^d \to \CC$ by $g := \Fourier^{-1} F$,
  then $f \!=\! \operatorname{Re}(g)|_{\Omega}$.
  We claim that $\| g \|_{C^\sigma(\R^d)} \!\leq\! (2 \pi)^{\sigma+2} (\delta + \| f \|_{B^\sigma(\Omega)})$;
  once this is shown, it follows that the same holds for $\operatorname{Re}(g)$, and this then
  implies $\| f \|_{C^\sigma(\Omega)} \leq (2 \pi)^{\sigma + 2} (\delta + \| f \|_{B^\sigma(\Omega)})$
  for arbitrary $\delta > 0$, from which the claim of the lemma follows.

  We first consider the case where $\sigma \in \N$.
  Note that we can ``differentiate under the integral'' in the formula
  $g(x) = \int_{\R^d} F(\xi) \, e^{2 \pi i \langle x,\xi \rangle} \, d \xi$
  to see for any multiindex $\alpha \in \N_0^d$ with $|\alpha| \leq \sigma$ that
  \begin{equation}
    \partial^\alpha g(x)
    = \int_{\R^d}
        F(\xi) \, (2 \pi i \xi)^\alpha \, e^{2 \pi i \langle x,\xi \rangle}
      \, d \xi
    \label{eq:DerivativeFourierRepresentation}
  \end{equation}
  and thus
  \begin{align*}
    |\partial^\alpha g(x)|
    & \leq (2 \pi)^{|\alpha|}
           \int_{\R^d}
             |F(\xi)| \, |\xi^\alpha|
           \, d \xi \\
    & \leq (2 \pi)^{\sigma}
           \int_{\R^d}
             |F(\xi)| \, (1 + |\xi|)^{\sigma}
           \, d \xi \\
    & \leq (2 \pi)^{\sigma} (\delta + \| f \|_{B^\sigma(\Omega)})
  \end{align*}
  for all $x \in \R^d$.
  This implies that the stated estimate holds for $\sigma \in \N$.

  \smallskip{}

  Finally, we consider the case where $\sigma \in (0,\infty) \setminus \N$.
  Write $\sigma = k + r$ with $k \in \N_0$ and $r \in [0,1)$.
  By ``differentiating under the integral'' as above, we see for $\alpha \in \N_0^d$
  with $|\alpha| \leq k$ that \Cref{eq:DerivativeFourierRepresentation} holds
  and that
  \begin{equation}
    \| g \|_{C^k(\R^d)}
    \leq (2 \pi)^k \cdot (\delta + \| f \|_{B^\sigma(\Omega)}).
    \label{eq:BarronCSigmaNormNonHoelderPart}
  \end{equation}
  Next, note that $|e^{i t} - e^{i s}| \leq 2$ and $|e^{i t} - e^{i s}| \leq |t - s|$,
  which implies that $|e^{i t} - e^{i s}| \leq 2^{1 - r} |t-s|^r$ for $t,s \in \R$.
  Therefore,
  \begin{align*}
    |\partial^\alpha g(x) - \partial^\alpha g(y)|
    & \leq (2 \pi)^{|\alpha|}
           \int_{\R^d}
             |F(\xi)|
             \cdot |\xi^\alpha|
             \cdot |e^{2 \pi i \langle x, \xi \rangle} - e^{2 \pi i \langle y, \xi \rangle}|
           \, d \xi \\
    & \leq (2 \pi)^{\sigma} 2^{1 - r}
           \int_{\R^d}
             |F(\xi)|
             \cdot |\xi^\alpha|
             \cdot (2 \pi \cdot |\langle x-y, \xi \rangle|)^{r}
           \, d \xi \\
    & \leq (2 \pi)^{\sigma} (2 \pi)^{1 - r} (2 \pi)^r
           \int_{\R^d}
             |F(\xi)| \cdot (1 + |\xi|)^{k} |x-y|^r |\xi|^r
           \, d \xi \\
    & \leq (2 \pi)^{\sigma + 1}
           |x-y|^r
           \int_{\R^d}
             |F(\xi)| \cdot (1 + |\xi|)^{\sigma}
           \, d \xi \\
    & \leq (2 \pi)^{\sigma + 1} (\delta + \| f \|_{B^\sigma(\Omega)}) \cdot |x-y|^r
  \end{align*}
  for all $x,y \in \R^d$, meaning
  $|g|_{C^\sigma(\R^d)} \leq (2 \pi)^{\sigma + 1} (\delta + \| f \|_{B^\sigma(\Omega)})$.
  Combined with \Cref{eq:BarronCSigmaNormNonHoelderPart}, this implies that
  $\|g\|_{C^\sigma(\R^d)} \leq (2 \pi)^{\sigma + 2} \cdot (\delta + \| f \|_{B^\sigma(\Omega)})$,
  as claimed.
\end{proof}

We now deduce the upper bound for the $L^\infty$ sampling number for the Barron space
$B^{\sigma}(\Omega)$ from analogous bounds for $C^\sigma(\Omega)$.

\begin{proposition}\label{prop:LInftyUpperBound}
  Let $d \in \N$ and $\sigma \in (0,\infty)$, and let $\Omega := [-\frac{1}{2}, \frac{1}{2}]^d$.
  There then exists a constant $C = C(d,\sigma) > 0$ such that for every $m \in \N$, we have
  \[
    s_m \bigl(B^\sigma(\Omega); L^\infty(\Omega)\bigr)
    \leq C \cdot m^{-\sigma/d}
    .
  \]
\end{proposition}

\begin{proof}
  \cite[Theorem~2]{KriegSonnleitnerRandomPointsForApproximatingSobolevFunctions} shows
  that there exists a constant $C_1 = C_1(d,\sigma) > 0$ such that for every $m \in \N$,
  there exist points $x_1,\dots,x_m \in \Omega$ and an operator $\varphi : \R^m \to L^\infty(\Omega)$
  such that for every $f \in C^s(\Omega)$ with $\| f \|_{C^s(\Omega)} \leq 1$, we have
  $\| f - \varphi(f(x_1),\dots,f(x_m)) \|_{L^\infty(\Omega)} \leq C_1 \cdot m^{-\sigma/d}$.
  Now, \Cref{lem:BarronIntoHoelderEmbedding} yields a constant $C_2 = C_2(d,\sigma) > 0$
  such that every $f \in B^{\sigma}(\Omega)$ with $\| f \|_{B^\sigma(\Omega)} \leq 1$
  satisfies $\| f \|_{C^s(\Omega)} \leq C_2$.
  Now, define
  \[
    T : \quad
    \R^m \to L^\infty(\Omega), \quad
    y \mapsto C_2 \cdot \varphi(y / C_2)
    .
  \]
  Then, if $\| f \|_{B^\sigma(\Omega)} \leq 1$, we know that $g := f / C_2$
  satisfies $\| g \|_{C^s (\Omega)} \leq 1$ and thus
  \[
    \| f - T(f(x_1),\dots,f(x_m)) \|_{L^\infty}
    = C_2 \, \| g - \varphi(g(x_1),\dots,g(x_m)) \|_{L^\infty}
    \leq C_2 C_1 \, m^{-\sigma / d}
    .
  \]
  This easily implies the claim.
\end{proof}

\subsection{Upper bounds for \texorpdfstring{$L^p$}{Lᷮ} sampling numbers}
\label{sub:LpUpperBounds}

To get upper bounds for the $L^p$ sampling numbers,
we will use the following general technical result.

\begin{lemma}\label{lem:SpecialInterpolationLemma}
  Let $(\Omega,\mu)$ be a measure space, let $\eps, \delta \geq 0$,
  and let $f,g,h : \Omega \to \R$ be measurable with
  \[
    \| f - g \|_{L^2(\mu)} \leq \eps
    \qquad \text{and} \qquad
    \| f - h \|_{L^\infty(\mu)} \leq \delta
    .
  \]
  Then the function
  \[
    F = F_{g,h,\delta} : \quad
    \Omega \to \R, \quad
    F(x) := \begin{cases}
              h(x) + \delta, & \text{if } g(x) \geq h(x) + \delta, \\
              h(x) - \delta, & \text{if } g(x) \leq h(x) - \delta, \\
              g(x), & \text{otherwise}
            \end{cases}
  \]
  satisfies for any $p \in [2,\infty]$ the estimate
  \[
    \| f - F \|_{L^p(\mu)}
    \leq 2 \, \eps^{\frac{2}{p}} \, \delta^{1 - \frac{2}{p}}
    .
  \]
\end{lemma}

\begin{proof}
  By changing $h$ (and thus possibly also $F$) on a null-set,
  we can assume that $|f(x) - h(x)| \leq \delta$ for all $x \in \Omega$.
  Now, we claim for any $x \in \Omega$ that $|f(x) - F(x)| \leq |f(x) - g(x)|$
  and furthermore $|f(x) - F(x)| \leq 2 \delta$.
  To prove this, we distinguish three cases:

  \noindent
  \textbf{Case 1 ($g(x) \geq h(x) + \delta$):}
  In this case, we see on the one hand that
  \[
    |f(x) - F(x)|
    = |f(x) - (h(x) + \delta)|
    \leq |f(x) - h(x)| + |-\delta|
    \leq 2 \delta ,
  \]
  which proves the second claimed property.
  On the other hand, we see because of $|f(x) - h(x)| \leq \delta$ that
  $g(x) \geq h(x) + \delta \geq f(x)$ and thus
  \begin{align*}
    |f(x) - g(x)|
    & = g(x) - f(x) \\
    & \geq h(x) + \delta - f(x)
      = |h(x) + \delta - f(x)| \\
    & = |F(x) - f(x)|
      = |f(x) - F(x)| ,
  \end{align*}
  which proves the first claimed property.

  \medskip{}

  \noindent
  \textbf{Case 2 ($g(x) \leq h(x) - \delta$):}
  In this case, we see on the one hand that
  \[
    |f(x) - F(x)|
    = |f(x) - (h(x) - \delta)|
    \leq |f(x) - h(x)| + |\delta|
    \leq 2 \delta ,
  \]
  which proves the second claimed property.
  On the other hand, we see because of $|f(x) - h(x)| \leq \delta$ that
  $g(x) \leq h(x) - \delta \leq f(x)$ and thus
  \begin{align*}
    |f(x) - g(x)|
    & = f(x) - g(x) \\
    & \geq f(x) - (h(x) - \delta) \\
    & =    |f(x) - (h(x) - \delta)|
      =    |f(x) - F(x)| ,
  \end{align*}
  which proves the first claimed property.

  \medskip{}

  \noindent
  \textbf{Case 3 (the remaining case):} In this case, $|g(x) - h(x)| < \delta$,
  and this easily implies that $|f(x) - F(x)| = |f(x) - g(x)| < \delta \leq 2 \delta$.
  Hence, both claimed estimates are valid in this case.

  \medskip{}

  Overall, we thus see $\| f - F \|_{L^2(\mu)} \leq \| f - g \|_{L^2(\mu)} \leq \eps$
  and $\| f - F \|_{L^\infty (\mu)} \leq 2 \delta$.
  Now, let $p \in [2, \infty]$ be arbitrary and set $\theta = 2/p$, noting that then
  $\frac{1}{p} = \frac{\theta}{2} + \frac{1-\theta}{\infty}$.
  Thus, combining the previous estimates with a standard interpolation estimate for $L^p$ norms
  (see for instance \cite[Proposition~6.10]{FollandRA}), we see that
  \[
    \| f - F \|_{L^p(\mu)}
    \leq \| f - F \|_{L^2(\mu)}^{\theta} \cdot \| f - F \|_{L^\infty(\mu)}^{1 - \theta}
    \leq \eps^{\theta} \cdot (2 \delta)^{1-\theta}
    \leq 2 \cdot \eps^{\theta} \cdot \delta^{1-\theta}
    ,
  \]
  which easily implies the claim of the lemma.
\end{proof}

\begin{corollary}\label{cor:SamplingNumberInterpolation}
  Let $\Omega \subset \R^d$ be compact and let $V \hookrightarrow C(\Omega)$.
  Then we have for all $p \in [2,\infty]$ and $m \in \N$ that
  \[
    s_{2m} (V;L^p)
    \leq 2 \cdot [s_m(V; L^2)]^{\frac{2}{p}} \cdot [s_m(V;L^\infty)]^{1 - \frac{2}{p}}
    .
  \]
\end{corollary}

\begin{proof}
  Let $m \in \N$ and $\gamma > 0$ be arbitrary.
  Set $\eps := s_m (V; L^2) + \gamma$ and $\delta := s_m(V;L^\infty) + \gamma$.
  By definition, there exist sampling points $x_1,\dots,x_m \in \Omega$ and a reconstruction map
  $T_2 : \R^m \to L^2(\Omega)$ such that for all $f \in V$ with $\| f \|_V \leq 1$, we have
  $\| f - T_2(f(x_1),\dots,f(x_m)) \|_{L^2} \leq \eps$.
  Likewise, there exist sampling points $x_{m+1},\dots,x_{2m} \in \Omega$ and a reconstruction map
  $T_\infty : \R^m \to L^\infty(\Omega)$ satisfying
  $\| f - T_\infty(f(x_{m+1}), \dots,f(x_{2m})) \|_{L^\infty(\Omega)} \leq \delta$
  for all $f \!\in\! V$ with $\| f \|_{V} \leq 1$.

  Now, let $p \in [2,\infty]$ be arbitrary and set
  \[
    T : \quad
    \R^{2 m} \to L^p(\Omega), \quad
    (y_1,\dots,y_{2 m}) \mapsto F_{T_2(y_1,\dots,y_m), T_\infty(y_{m+1},\dots,y_{2m}), \delta}
    ,
  \]
  with $F_{g,h,\delta}$ as defined in \Cref{lem:SpecialInterpolationLemma}.
  Now, let $f \in V$ with $\| f \|_V \leq 1$ be arbitrary and set
  $g := T_2(f(x_1),\dots,f(x_m))$ and $h := T_\infty(f(x_{m+1}),\dots,f(x_{2m}))$.
  As seen above, we then have $\| f - g \|_{L^2(\Omega)} \leq \eps$
  and $\| f - h \|_{L^\infty(\Omega)} \leq \delta$.
  Therefore, \Cref{lem:SpecialInterpolationLemma} shows that
  \begin{align*}
    \| f - T(f(x_1),\dots,f(x_{2m})) \|_{L^p(\Omega)}
    & = \| f - F_{g,h,\delta} \|_{L^p(\Omega)} \\
    & \leq 2 \, \eps^{\frac{2}{p}} \delta^{1 - \frac{2}{p}} \\
    & =    2
           \cdot [s_m(V;L^2) + \gamma]^{\frac{2}{p}}
           \cdot [s_m (V;L^\infty) + \gamma]^{1 - \frac{2}{p}}
    .
  \end{align*}
  Thus,
  \(
    s_{2m}(V; L^p)
    \leq 2
         \cdot [s_m(V;L^2) + \gamma]^{\frac{2}{p}}
         \cdot [s_m (V;L^\infty) + \gamma]^{1 - \frac{2}{p}}
    .
  \)
  Since this holds for arbitrary $\gamma > 0$, we are done.
\end{proof}

As a consequence, we now obtain upper bounds for the $L^p$-sampling numbers of the Barron space
$B^{\sigma}(\Omega)$ for general $p \in [1,\infty]$.

\begin{theorem}\label{thm:GeneralUpperBound}
  Let $d \in \N$ and $\Omega := [-\frac{1}{2}, \frac{1}{2}]^d$.
  Given $\sigma \in (0,\infty)$, there exists a constant $C = C(d,\sigma) \geq 1$ such that
  for every $p \in [1,\infty]$ and $m \in \N$, we have
  \[
    s_m \bigl(B^\sigma (\Omega); L^p(\Omega)\bigr)
    \leq \begin{cases}
           C \cdot \big( [\ln(e + m)]^4 \big/ m \big)^{\frac{1}{2} + \frac{\sigma}{d}},
           & \text{if } p \in [1,2], \\[0.2cm]
           C \cdot [\ln(e + m)]^{(4 + 8 \frac{\sigma}{d}) / p}
             \cdot m^{-(\frac{1}{p} + \frac{\sigma}{d})},
           & \text{if } p \in [2,\infty]. \\
         \end{cases}
  \]
\end{theorem}

\begin{proof}
  Let $\lambda_0 := \frac{1}{2} + \frac{\sigma}{d}$.
  \Cref{cor:L2SamplingNumberUpperBound} yields $C_1 = C_1(d,\sigma) \geq 1$ such that
  \[
    s_m \bigl(B^\sigma(\Omega); L^2(\Omega)\bigr)
    \leq C_1 \cdot \Big( \ln^4(e + m) \Big/ m \Big)^{\lambda_0}
    \qquad \forall \, m \in \N .
  \]
  Similarly, \Cref{prop:LInftyUpperBound} yields $C_2 = C_2(d,\sigma) \geq 1$ such that
  \[
    s_m \big( B^\sigma(\Omega); L^\infty(\Omega) \big)
    \leq C_2 \cdot m^{-\sigma / d}
    .
  \]
  Define $C_3 := \max \{ C_1, C_2 \}$
  and $C := 2 \cdot 4^{1 + \frac{\sigma}{d}} \cdot C_3$.
  We now distinguish two cases regarding $p \in [1,\infty]$.

  \smallskip{}

  \noindent
  \textbf{Case~1 ($p \in [1,2]$):} For $p \in [1,2]$, Hölder's inequality shows because of
  $\lebesgue(\Omega) = 1$ that $L^2(\Omega) \hookrightarrow L^p(\Omega)$ with
  $\| \bullet \|_{L^p(\Omega)} \leq \| \bullet \|_{L^2(\Omega)}$.
  This easily implies
  \[
    s_m \big( B^\sigma(\Omega); L^p(\Omega) \big)
    \leq s_m \big( B^\sigma(\Omega); L^2(\Omega) \big)
    \leq C_3 \cdot \Big( \ln^4(e + m) \Big/ m \Big)^{\lambda_0}
    .
  \]
  Since $C_3 \leq C$, this implies the claim in this case.

  \medskip{}

  \noindent
  \textbf{Case~2 ($p \in [2,\infty]$):}
  We first consider the case $m \!\in\! \N_{\geq 2}$ and set $n := \lfloor m / 2 \rfloor \!\in\! \N$,
  noting that $\frac{m}{2} \leq 1 + n \leq 2 n$ and thus $m \geq n \geq m/4$.
  Then, \Cref{cor:SamplingNumberInterpolation} (applied with $n$ instead of $m$) shows
  \begin{align*}
    s_m \big( B^\sigma(\Omega); L^p(\Omega) \big)
    & \leq s_{2 n} \big( B^\sigma(\Omega); L^p(\Omega) \big) \\
    & \leq 2
           \cdot \bigl[ s_{n} \big( B^\sigma(\Omega); L^2(\Omega) \big) \bigr]^{\frac{2}{p}}
           \cdot \bigl[ s_{n} \big( B^\sigma(\Omega); L^\infty(\Omega) \big) \bigr]^{1 - \frac{2}{p}} \\
    & \leq 2 C_3
           \cdot \Big( \ln^4 (e + n) \Big/ n  \Big)^{2 \lambda_0 / p}
           \cdot n^{- \frac{\sigma}{d} (1 - \frac{2}{p})} \\
    & =    2 C_3 \cdot \ln^{8\lambda_0 / p} (e + n) \cdot n^{-(\frac{1}{p} + \frac{\sigma}{d})} \\
    & \leq 2 \cdot 4^{\frac{1}{p} + \frac{\sigma}{d}} \cdot C_3
           \cdot \ln^{8\lambda_0 / p} (e + m) \cdot m^{-(\frac{1}{p} + \frac{\sigma}{d})} \\
    & \leq C \cdot \ln^{8\lambda_0 / p} (e + m) \cdot m^{-(\frac{1}{p} + \frac{\sigma}{d})} 
    .
  \end{align*}
  This proves the claim for $p \in [2,\infty]$ and $m \in \N_{\geq 2}$.

  Finally, if $m = 1$, simply recall from \Cref{lem:BarronEmbedsIntoLInfty} that
  $\| f \|_{L^p} \leq \| f \|_{L^\infty} \leq 1$ for all $f \in B^\sigma(\Omega)$
  with $\| f \|_{B^\sigma} \leq 1$; by taking $T \equiv 0$ as the reconstruction map,
  this easily implies
  \[
    s_m \big( B^\sigma(\Omega); L^p(\Omega) \big)
    \leq 1
    \leq \ln^{8 \lambda_0 / p} (e + m) \cdot m^{-(\frac{1}{p} + \frac{\sigma}{d})}
    \leq C \cdot \ln^{8 \lambda_0 / p} (e + m) \cdot m^{-(\frac{1}{p} + \frac{\sigma}{d})}
    .
  \]
  Thus, the claim holds for all $m \in \N$.
\end{proof}

\section{Lower bounds}%
\label{sec:LowerBounds}

In the previous section, we showed that a certain \emph{deterministic} choice of sampling points
and reconstruction algorithm yields a certain decay of the recovery error as the number of sampling
points tends to $\infty$.
In this section, we show that this rate is optimal (possibly up to log factors),
even in the class of \emph{randomized} algorithms.
To formally prove this, let us first clarify the notion of (deterministic and randomized)
algorithms based on point samples.

\begin{definition}\label{def:AlgorithmsBasedOnPointSamples}
  Let $\emptyset \neq \Omega \subset \R^d$ be arbitrary and $p \in [1,\infty]$.
  A \emph{deterministic ($L^p$-valued) algorithm based on $m$ point samples} is any map of the form
  \[
    A = T \circ \Phi ,
  \]
  where $\Phi : C(\Omega) \to \R^m, f \mapsto (f(x_1),\dots,f(x_m))$ for certain $x_1,\dots,x_m \in \Omega$
  and where $T : \R^m \to L^p(\Omega)$ is arbitrary.
  We then define the \emph{worst-case error} of the algorithm $A$ on $B^\sigma(\Omega)$ as
  \[
    E_p (A; B^\sigma(\Omega))
    := \sup_{\| f \|_{B^\sigma(\Omega)} \leq 1} \| f - A(f) \|_{L^p}
    ,
  \]
  and we write $\Alg_m^{\mathrm{det}} (L^p)$ for the set of all these algorithms.

  A \emph{randomized ($L^p$-valued) algorithm using at most $m$ point samples in expectation} is a family
  $\mathbf{A} = (A_\omega)_{\omega \in \Omega}$ indexed by a probability space $(\mathbf{X}, \PP)$
  such that the following properties hold:
  \begin{enumerate}
    \item There exists a measurable map $\mathbf{m} : \mathbf{X} \to \N$ satisfying
          \begin{itemize}
            \item $A_\omega \in \Alg_{\mathbf{m}(\omega)}^{\mathrm{det}}(L^p)$ for all $\omega \in \mathbf{X}$, and
            \item $\EE [\mathbf{m}(\omega)] \leq m$.
          \end{itemize}

    \item The map $\omega \mapsto \| f - A_\omega (f) \|_{L^p}$ is measurable
          for every $f \in B^\sigma(\Omega)$ with $\| f \|_{B^\sigma(\Omega)} \leq 1$.
  \end{enumerate}
  The \emph{worst-case expected error} of $\mathbf{A}$ on $B^\sigma(\Omega)$ is defined as
  \[
    E_p(\mathbf{A}; B^\sigma(\Omega))
    := \sup_{\| f \|_{B^\sigma(\Omega)} \leq 1}
         \EE \| f - A_\omega(f) \|_{L^p}
    ,
  \]
  and we write $\Alg_m^{\mathrm{MC}}(L^p)$ for the class of all such algorithms;
  here, the ``MC'' stand for ``Monte Carlo''.
\end{definition}

\begin{remark*}
  \begin{enumerate}
    \item We note that $\Alg_m^{\mathrm{det}}(L^p) \subset \mathrm{Alg}_m^{\mathrm{MC}}(L^p)$
          (up to abuse of notation), since one can interpret every deterministic algorithm
          as a randomized algorithm over the trivial probability space $\mathbf{X} = \{ 0 \}$.

    \item The usage of the word ``algorithm'' is \emph{not} to be understood as implying some
          form of (Turing) computability; as stated above, the map $T$ in $A = T \circ \Phi$
          can be \emph{arbitrary}, even non-computable.

          We remark, however, that the algorithm that we construct in our upper bounds
          (see e.g.\ \Cref{thm:UpperBoundL2}) \emph{is} in fact computable.
  \end{enumerate}
\end{remark*}

Our goal in this section is to prove the following lower bound:

\begin{theorem}\label{thm:LowerBound}
  For any $d \in \N$, $\sigma \in (0,\infty)$, and $p \in [1,\infty]$,
  there exists a constant $C = C(d,\sigma,p) > 0$
  such that for $\Omega = [-\frac{1}{2}, \frac{1}{2}]^d$ and any randomized $L^p$-valued
  algorithm $\mathbf{A}$ using at most $m \in \N$ point samples in expectation, we have
  \[
    E_p (\mathbf{A}; B^\sigma(\Omega))
    \geq \begin{cases}
           C \cdot m^{-(1/2 + \sigma/d)}, & \text{if } p \leq 2, \\[0.2cm]
           C \cdot m^{-(1/p + \sigma/d)}, & \text{if } p >    2.
         \end{cases}
  \]
\end{theorem}

\subsection{Preparations}
\label{sub:Preparations}

In this subsection, we collect a few technical results which we will then use to
prove \Cref{thm:LowerBound} in the following subsection.

\begin{lemma}\label{lem:BumpSumNormEstimate}
  Let $d \in \N$ and $\sigma \in (0,\infty)$, set $\Omega := [-\frac{1}{2}, \frac{1}{2}]^d$, and fix
  $\varphi \in C_c^\infty(\R^d)$ with $\varphi \not\equiv 0$ and $\supp \varphi \subset (0,1)^d$.
  Then there exists a constant $C = C(\varphi,d,\sigma)$ with the following property:

  Let $N \in \N$, set $\Lambda_N^\ast := \{ 0,\dots,N-1 \}^d$, and define
  \[
    \psi_{N,\Lambda,\theta}
    := \sum_{n \in \Lambda}
         \theta_n \cdot \varphi\bigl(N \cdot (\bullet + (\tfrac{1}{2},\dots,\tfrac{1}{2}) - \tfrac{n}{N})\bigr)
  \]
  for $\Lambda \subset \Lambda_N^\ast$ and $\theta \in \{ \pm 1 \}^{\Lambda_N^\ast}$.
  Then we have
  \[
    \| \psi_{N,\Lambda,\theta} \|_{B^\sigma(\Omega)}
    \leq C \cdot N^\sigma \cdot |\Lambda|^{1/2}
    .
  \]
\end{lemma}

\begin{proof}
  For brevity, set $\varphi_N (x) := \varphi(N x)$ and note with the usual translation operator
  $T_x$ defined by $T_x f (y) = f(y - x)$ that
  \[
    \psi_{N,\Lambda,\theta}
    = \sum_{n \in \Lambda}
        \theta_n \, T_{\frac{n}{N} - (\frac{1}{2},\dots,\frac{1}{2})} \varphi_N
    .
  \]
  By standard properties of the Fourier transform, we see that
  $\psi_{N,\Lambda,\theta} \in C_c^\infty(\R^d)$ with Fourier transform
  \[
    \Fourier \psi_{N,\Lambda,\theta}
    = \sum_{n \in \Lambda}
        \theta_n \,
        e^{-2 \pi i \langle \frac{n}{N} - (\frac{1}{2},\dots,\frac{1}{2}), \xi \rangle} \,
        \widehat{\varphi_N} (\xi)
    = e^{\pi i \langle (1,\dots,1), \xi \rangle} \,
      N^{-d} \,
      \widehat{\varphi} (\tfrac{\xi}{N}) \,
      g_{\Lambda,\theta} (\tfrac{\xi}{N})
  \]
  for
  \(
    g_{\Lambda,\theta} (\xi)
    := \sum_{n \in \Lambda}
         \theta_n \, e^{- 2 \pi i \langle n, \xi \rangle}
    .
  \)
  Note that $g$ is $\Z^d$-periodic.

  Next, since $\widehat{\varphi} \in \Schwartz(\R^d)$,
  there exists a constant $C_1 = C_1(\varphi,d,\sigma) > 0$ such that
  $|\widehat{\varphi}(\xi)| \leq C_1 \cdot (1 + |\xi|)^{-(d+1+\sigma)}$ for all $\xi \in \R^d$.
  Furthermore, note that if $\ell \in \Z^d$ and $\eta \in [0,1]^d$, then
  \(
    1 + |\ell|
    \leq 1 + |\ell + \eta| + |-\eta|
    \leq 1 + \sqrt{d} + |\ell + \eta|
    \leq (1 + \sqrt{d}) \cdot (1 + |\ell + \eta|)
    ,
  \)
  and thus $(1 + |\ell+\eta|)^{-(d+1)} \leq C_2 \cdot (1 + |\ell|)^{-(d+1)}$
  for $C_2 := (1 + \sqrt{d})^{d+1}$.
  Therefore, we see
  \[
    \sum_{\ell \in \Z^d}
      (1 + |\ell + \eta|)^{-(d+1)}
    \leq C_2 \cdot \sum_{\ell \in \Z^d} (1 + |\ell|)^{-(d+1)}
    =:   C_3
    \qquad \forall \, \eta \in [0,1]^d ,
  \]
  where $C_3 = C_3(d) < \infty$.

  Combining these observations, we see
  \begin{align*}
    & \int_{\R^d}
        (1 + |\xi|)^\sigma \cdot |\Fourier \psi_{N,\Lambda,\theta}(\xi)|
      \, d \xi \\
    & = \int_{\R^d}
          N^{-d}
          \cdot \bigl(1 + N \cdot |\tfrac{\xi}{N}|\bigr)^\sigma
          \cdot |\widehat{\varphi} (\tfrac{\xi}{N})|
          \cdot |g_{\Lambda,\theta} (\tfrac{\xi}{N})|
        \, d \xi \\
    & = \int_{\R^d}
          (1 + N \cdot |\eta|)^\sigma
          \cdot |\widehat{\varphi}(\eta)|
          \cdot |g_{\Lambda,\theta}(\eta)|
        \, d \eta \\
    & \leq C_1 \cdot N^\sigma
           \cdot \int_{\R^d}
                   (1 + |\eta|)^{-(d+1)} \cdot |g_{\Lambda,\theta}(\eta)|
                 \, d \eta \\
    & =    C_1 \cdot N^\sigma
           \cdot \sum_{\ell \in \Z^d}
                 \int_{[0,1]^d}
                   (1 + |\eta + \ell|)^{-(d+1)} \cdot |g_{\Lambda,\theta}(\eta + \ell)|
                 \, d \eta \\
    ({\scriptstyle{g_{\Lambda,\theta} \text{ is $\Z^d$-periodic}}})
    & =    C_1 \cdot N^\sigma
           \int_{[0,1]^d}
             |g_{\Lambda,\theta}(\eta)|
             \sum_{\ell \in \Z^d}
               (1 + |\eta + \ell|)^{-(d+1)}
           \, d \eta \\
    & \leq C_1 C_3 \cdot N^\sigma \cdot \| g_{\Lambda,\theta} \|_{L^1([0,1]^d)} \\
    & \leq C_1 C_3 \cdot N^\sigma \cdot \| g_{\Lambda,\theta} \|_{L^2([0,1]^d)} \\
    & =    C_1 C_3 \cdot N^\sigma \cdot \bigg( \sum_{n \in \Lambda} |\theta_n|^{2} \bigg)^{1/2}
      =    C_1 C_3 \cdot N^\sigma \cdot |\Lambda|^{1/2}
    .
  \end{align*}
  Finally, by Fourier inversion (which is valid since
  $\psi_{N,\Lambda,\theta} \in C_c^\infty (\R^d) \subset \Schwartz(\R^d)$), we see
  \[
    \psi_{N,\Lambda,\theta} (x)
    = \Fourier^{-1} [\Fourier \psi_{N,\Lambda,\theta}](x)
    = \int_{\R^d}
         \Fourier \psi_{N,\Lambda,\theta} (\xi)
         e^{2 \pi i \langle x,\xi \rangle}
      \, d \xi
    .
  \]
  By definition of the Barron norm, this implies
  $\| \psi_{N,\Lambda,\theta} \|_{B^\sigma(\Omega)} \leq C_1 C_3 \cdot N^\sigma \cdot |\Lambda|^{1/2}$,
  which implies the claim.
\end{proof}

\subsection{Proof of Theorem~\ref{thm:LowerBound}}%
\label{sub:LowerBoundProof}

We remark that the employed proof idea (deriving lower bounds for randomized algorithms
using lower bounds for a suitably chosen average-case error of deterministic algorithms)
is well-known in the information-based complexity community;
see for instance \cite{HeinrichRandomApproximation}.

Let $d \in \N$, $\sigma \in (0,\infty)$ and $p \in [1,\infty]$ be arbitrary.

\noindent
\textbf{Step~1 (Preparation):}
Fix $\varphi \in C_c^\infty(\R^d) \strut$ with $\supp \varphi \subset (0,1)^d$ and $\varphi \not\equiv 0$.
Let $m \in \N$ and choose $N := \big\lceil (4 m)^{1/d} \big\rceil \strut$, and
\[
  k := \begin{cases}
         1,   & \text{if } p > 2, \\
         N^d, & \text{if } p \leq 2 ,
       \end{cases}
\]
and set $\Lambda_N^\ast := \{ 0,1,\dots,N-1 \}^d$ and
$S_k := \{ \Lambda \subset \Lambda_N^\ast \colon |\Lambda| = k \}$.
For $\Lambda \in S_k$ and $\theta \in \{ \pm 1 \}^{\Lambda_N^\ast}$, define
\[
  \gamma_{\Lambda,\theta}
  := \frac{1}{C_1 \cdot N^\sigma \cdot k^{1/2}} \cdot \psi_{N,\Lambda,\theta}
  ,
\]
with $C_1 = C_1(d,\sigma,\varphi) > 0$ and $\psi_{N,\Lambda,\theta}$
as in \Cref{lem:BumpSumNormEstimate}.
By that lemma, we have $\| \gamma_{\Lambda,\theta} \|_{B^\sigma(\Omega)} \leq 1$.

\medskip{}

\noindent
\textbf{Step~2 (Average-case error bound for deterministic algorithms):}
Let $x_1,\dots,x_{2m} \in \Omega$ and $R : \R^{2m} \!\!\to\! L^p(\Omega)$ be arbitrary
but fixed, and let $A \!:=\! R \circ \Phi$ for $\Phi \!:\! C(\Omega) \!\to\! \R^{2m}$ given by
$\Phi(f) = (f(x_1),\dots,f(x_{2m}))$.
Define
\[
  I := \big\{
         n \in \{ 0,\dots,N-1 \}^d
         \colon
         \forall \, i \in \FirstN{2m} : x_i \notin \Omega_n
       \big\}
  ,
\]
where we set
\[
  \Omega_n
  := (0, \tfrac{1}{N})^d + \tfrac{n}{N} - (\tfrac{1}{2},\dots,\tfrac{1}{2})
  \quad \text{for} \quad
  n \in \{ 0,1,\dots,N-1 \}^d
  .
\]
Note for each $n \in I^c = \Lambda_N^\ast \setminus I$ that there exists $i_n \in \FirstN{2m}$
such that $x_{i_n} \in \Omega_n$.
Since $\Omega_n \cap \Omega_m = \emptyset$ for $n \neq m$,
the map $I^c \to \FirstN{2m}, n \mapsto i_n$ is injective,
which implies that $|I^c| \leq 2 m$, and thus
\begin{equation}
  |I|
  \geq N^d - 2 m
  \geq N^d - \frac{N^d}{2}
  =    \frac{1}{2} N^d
  =    \frac{1}{2} |\Lambda_N^\ast|
  .
  \label{eq:ICardinalityLowerBound}
\end{equation}

Next, define
\[
  \iota : \quad
  \{ \pm 1 \}^{\Lambda_N^\ast} \to \{ \pm 1 \}^{\Lambda_N^\ast}, \quad
  (\iota \theta)_n = \begin{cases}
                       - \theta_n , & \text{if } n \in I, \\
                       + \theta_n , & \text{if } n \notin I.
                     \end{cases}
\]
Note that $\iota (\iota \theta) = \theta$, which implies that $\iota$ is bijective.
Next, note that
\[
  \varphi_n
  := \varphi \big( N \cdot (\bullet + (\tfrac{1}{2},\dots,\tfrac{1}{2}) - \tfrac{n}{N}) \big)
\]
satisfies $\supp \varphi_n \subset \Omega_n$ and hence $\varphi_n (x_i) = 0$ for all $i \in \FirstN{2m}$,
provided that $n \in I$.
Therefore, $\theta_n \cdot \varphi_n (x_i) = (\iota \theta)_n \cdot \varphi_n (x_i)$
for all $n \in \Lambda_N^\ast$ and $i \in \FirstN{2m}$, and thus
\[
  \Phi(\gamma_{\Lambda,\theta}) = \Phi(\gamma_{\Lambda,\iota \theta})
  \qquad \forall \, \Lambda \in S_k \text{ and } \theta \in \{ \pm 1 \}^{\Lambda_N^\ast}
  .
\]
Therefore, setting
\[
  \eps_{\Lambda,\theta}
  := \big\| \gamma_{\Lambda,\theta} - R (\Phi(\gamma_{\Lambda,\theta})) \big\|_{L^p}
  ,
\]
we see for any $\Lambda \in S_k$ that
\begin{align*}
  \eps_{\Lambda,\theta} + \eps_{\Lambda,\iota \theta}
  & = \big\| \gamma_{\Lambda,\theta} - R(\Phi(\gamma_{\Lambda,\theta})) \big\|_{L^p}
      + \big\| R(\Phi(\gamma_{\Lambda,\theta})) - \gamma_{\Lambda, \iota \theta} \big\|_{L^p} \\
  & \geq \| \gamma_{\Lambda,\theta} - \gamma_{\Lambda, \iota \theta} \|_{L^p} \\
  & =    \frac{1}{C_1 \cdot N^\sigma \cdot k^{1/2}}
         \cdot \bigg\|
                 \sum_{n \in \Lambda \cap I}
                   (\theta_n - (-\theta_n)) \, \varphi_n
               \bigg\|_{L^p(\Omega)}
  .
\end{align*}
Now, since the $(\varphi_n)_{n \in \Lambda_N^\ast}$ have disjoint supports
contained i n $\Omega$ and since
\(
  \| \varphi_n \|_{L^p}
  = \| \varphi(N \bullet) \|_{L^p}
  = \| \varphi \|_{L^p} \cdot N^{-d/p}
  ,
\)
it is easy to see
\begin{align*}
  \bigg\|
    \sum_{n \in \Lambda \cap I}
      (\theta_n - (-\theta_n)) \, \varphi_n
  \bigg\|_{L^p}
  & = \| \varphi \|_{L^p}
      \cdot N^{-d/p}
      \cdot \| (2 \, \theta_n)_{n \in \Lambda \cap I} \|_{\ell^p} \\
  & = 2 \, \| \varphi \|_{L^p}
      \cdot N^{-d/p}
      \cdot |\Lambda \cap I|^{1/p}
  ,
\end{align*}
where for $p = \infty$ we use the convention $0^0 = 0$.
Since $\iota : \{ \pm 1 \}^{\Lambda_N^\ast} \to \{ \pm 1 \}^{\Lambda_N^\ast}$ is bijective,
this implies
\begin{align*}
  \avsum_{\theta \in \{ \pm 1 \}^{\Lambda_N^\ast}}
    \eps_{\Lambda,\theta}
  & = \frac{1}{|\{ \pm 1 \}^{\Lambda_N^\ast}|}
        \sum_{\theta \in \{ \pm 1 \}^{\Lambda_N^\ast}}
          \eps_{\Lambda,\theta}
    = \frac{1}{2 \, |\{ \pm 1 \}^{\Lambda_N^\ast}|}
      \sum_{\theta \in \{ \pm 1 \}^{\Lambda_N^\ast}}
        (\eps_{\Lambda,\theta} + \eps_{\Lambda, \iota \theta}) \\
  & \geq \frac{1}{|\{ \pm 1 \}^{\Lambda_N^\ast}|}
         \sum_{\theta \in \{ \pm 1 \}^{\Lambda_N^\ast}}
           \frac{\| \varphi \|_{L^p}}{C_1 \cdot N^{\sigma + d/p}
           \cdot k^{1/2}}
           \cdot |\Lambda \cap I|^{1/p} \\
  & =     \frac{\| \varphi \|_{L^p}}{C_1}
          \cdot N^{-(\sigma + d/p)}
          \cdot k^{-1/2}
          \cdot |\Lambda \cap I|^{1/p}
  .
\end{align*}

Now, in case of $p > 2$, we have $k = 1$, so that for $\Lambda \in S_k$ (meaning $|\Lambda| = k = 1$),
we have $|\Lambda \cap I| \in \{ 0,1 \}$ and thus $|\Lambda \cap I|^{1/p} = |\Lambda \cap I|$
(also for $p = \infty$).
Therefore, in case of $p > 2$, we see thanks to \Cref{eq:ICardinalityLowerBound} that
\[
  \sum_{\Lambda \in S_k}
    |\Lambda \cap I|^{1/p}
  = \sum_{\Lambda \in S_k}
      |\Lambda \cap I|
  = \sum_{n \in \Lambda_N^\ast}
      |\{ n \} \cap I|
  = |I|
  \geq \frac{1}{2} |\Lambda_N^\ast|
  =    \frac{1}{2} |S_k|
  ,
\]
and thus $\avsum_{\Lambda \in S_k} |\Lambda \cap I|^{1/p} \geq \frac{1}{2}$.

If otherwise $p \leq 2$, then $k = N^d$, so that $S_k = \{ \Lambda_N^\ast \}$, and hence
\[
  \avsum_{\Lambda \in S_k} |\Lambda \cap I|^{1/p}
  = |I|^{1/p}
  \geq (N^d / 2)^{1/p}
  \geq \frac{1}{2} N^{d/p}
  ,
\]
again thanks to \Cref{eq:ICardinalityLowerBound}.
Overall, these considerations show that
\begin{align*}
  & \avsum_{\Lambda \in S_k, \theta \in \{ \pm 1 \}^{\Lambda_N^\ast}}
      \| \gamma_{\Lambda,\theta} - R(\Phi(\gamma_{\Lambda,\theta})) \|_{L^p} \\
  & = \avsum_{\Lambda \in S_k, \theta \in \{ \pm 1 \}^{\Lambda_N^\ast}}
        \eps_{\Lambda,\theta} \\
  & \geq 2 C_2 \cdot N^{-(\sigma + d/p)} \cdot k^{-1/2}
         \avsum_{\Lambda \in S_k}
           |\Lambda \cap I|^{1/p} \\
  & \geq \begin{cases}
           C_2 \cdot N^{-(\sigma + d/p)},                                                        & \text{if } p > 2, \\
           C_2 \cdot N^{-(\sigma + d/p)} \cdot N^{-d/2} N^{d/p} = C_2 \cdot N^{-(\sigma + d/2)}, & \text{if } p \leq 2 ,
         \end{cases}
\end{align*}
for $C_2 := \frac{\| \varphi \|_{L^p}}{2 \, C_1}$.

Overall, since $N \leq 1 + (4m)^{1/d} \leq 2 \cdot (4m)^{1/d} \leq 8 \, m^{1/d}$,
we have thus shown for every $A \in \Alg^{\mathrm{det}}_{2m} (L^p)$ that
\begin{equation}
  \avsum_{\Lambda \in S_k, \theta \in \{ \pm 1 \}^{\Lambda_N^\ast}}
    \| \gamma_{\Lambda,\theta} - A(\gamma_{\Lambda,\theta}) \|_{L^p}
  \geq \begin{cases}
         C_3 \cdot m^{-(1/p + \sigma/d)}, & \text{if } p > 2, \\
         C_3 \cdot m^{-(1/2 + \sigma/d)}, & \text{if } p \leq 2 ,
       \end{cases}
  \label{eq:DeterministicAverageCaseLowerBound}
\end{equation}
where $C_3 := 8^{-(\sigma + d/p)} \, C_2$ for $p > 2$ and $C_3 := 8^{-(\sigma + d/2)}$ for $p \leq 2$.

\medskip{}

\noindent
\textbf{Step~3 (Completing the proof):}
Let $\mathbf{A} = (A_\omega)_{\omega \in \mathbf{X}} \in \Alg^{\mathrm{MC}}_m (L^p)$ be arbitrary.
This means that there exist a probability space $(\mathbf{X}, \PP)$ and a measurable map
$\mathbf{m} : \mathbf{X} \to \N$ such that $\mathbf{A} = (A_\omega)_{\omega \in \mathbf{X}}$
with $A_\omega \in \Alg^{\mathrm{det}}_{\mathbf{m}(\omega)} (L^p)$ for all $\omega \in \mathbf{X}$
and such that $\EE [\mathbf{m}(\omega)] \leq m$.

Let $\mathbf{X}_0 := \{ \omega \in \mathbf{X} \colon \mathbf{m}(\omega) \leq 2 m \}$ and note
\[
  m
  \geq \EE [\mathbf{m}(\omega)]
  \geq \EE [2m \, \indicator_{\mathbf{X}_0^c}(\omega)]
  =    2m \cdot \PP(\mathbf{X}_0^c)
  ,
\]
which implies that $\PP(\mathbf{X}_0^c) \leq \frac{1}{2}$
and hence $\PP(\mathbf{X}_0) \geq \frac{1}{2}$.

Thus, we see from the definition of $E_p(\mathbf{A}; B^\sigma(\Omega))$,
because of $\| \gamma_{\Lambda,\theta} \|_{B^\sigma} \leq 1$,
and since $A_\omega \in \Alg_m^{\mathrm{det}}(L^p)$ for $\omega \in \mathbf{X}_0$,
so that \Cref{eq:DeterministicAverageCaseLowerBound} is applicable,
\begin{align*}
  E_p (\mathbf{A}; B^\sigma(\Omega))
  & = \sup_{\| f \|_{B^\sigma(\Omega)} \leq 1}
        \EE \| f - A_\omega(f) \|_{L^p} \\
  & \geq \avsum_{\Lambda \in S_k, \theta \in \{ \pm 1 \}^{\Lambda_N^\ast}}
           \EE \| \gamma_{\Lambda,\theta} - A_\omega(\gamma_{\Lambda,\theta}) \|_{L^p} \\
  & =    \EE
         \bigg[
           \avsum_{\Lambda \in S_k, \theta \in \{ \pm 1 \}^{\Lambda_N^\ast}}
             \| \gamma_{\Lambda,\theta} - A_\omega(\gamma_{\Lambda,\theta}) \|_{L^p}
         \bigg] \\
  & \geq \EE
         \bigg[
           \indicator_{\mathbf{X}_0}(\omega)
           \avsum_{\Lambda \in S_k, \theta \in \{ \pm 1 \}^{\Lambda_N^\ast}}
             \| \gamma_{\Lambda,\theta} - A_\omega(\gamma_{\Lambda,\theta}) \|_{L^p}
         \bigg] \\
  & \geq \EE
         \big[
           \indicator_{\mathbf{X}_0}(\omega)
           \cdot C_3
           \cdot m^{-(\sigma/d + 1 / \max\{ 2,p \})}
         \big] \\
  & \geq \frac{C_3}{2} \cdot m^{-(\sigma/d + 1 / \max \{ 2,p \})}
  .
\end{align*}
This completes the proof of \Cref{thm:LowerBound}.
\hfill$\square$

\section{Lower bounds for linear sampling numbers}%
\label{sec:LinearSamplingNumbersLowerBound}

In this section, we show based on observations and techniques
in \cite[Section~X]{BarronUniversalApproximationBoundsForSigmoidal} that the
\emph{linear} sampling numbers are subject to the curse of dimension,
in stark contrast to the case of the (general, non-linear) sampling numbers
that we studied in the previous sections.
In fact, for \emph{deterministic} linear reconstruction based on point samples,
the desired lower bound is an easy consequence of
\cite[Theorem~6]{BarronUniversalApproximationBoundsForSigmoidal};
but since we want to show that the same also holds for potentially randomized linear
reconstruction, we need to delve into the proof techniques
employed in \cite{BarronUniversalApproximationBoundsForSigmoidal}.

\begin{remark*}
  We remark that we derive these lower bounds only for the case where the reconstruction
  error is measured in $L^p$ with $p \geq 2$.
  We leave the case of $1 \leq p < 2$ as an interesting problem for future work.
\end{remark*}

Before we properly begin, let us define what we mean by ``randomized linear reconstruction''.

\begin{definition}\label{def:LinearAlgorithmsBasedOnPointSamples}
  Let $\emptyset \neq \Omega \subset \R^d$ be arbitrary and $p \in [1,\infty]$.
  A deterministic ($L^p$-valued) \emph{linear} algorithm based on $m$ point samples
  is any map of the form
  \[
    A = T \circ \Phi ,
  \]
  where $T : \R^m \to L^p(\Omega)$ is \emph{linear} and where $\Phi : C(\Omega) \to \R^m$
  is of the form $\Phi (f) = \big( f(x_1),\dots,f(x_m) \big)$ for certain $x_1,\dots,x_m \in \Omega$.
  We write $\Alg_m^{\mathrm{lin}, \mathrm{det}}(L^p)$ for the set of all these algorithms.

  A randomized ($L^p$-valued) \emph{linear} algorithm using at most $m$ point samples in expectation
  is a family $\mathbf{A} = (A_\omega)_{\omega \in \Omega}$ indexed by a probability space
  $(\mathbf{X}, \PP)$ such that the following properties hold:
  \begin{enumerate}
    \item There exists a measurable map $\m : \X \to \N$ satisfying
          \begin{itemize}
            \item $A_\omega \in \Alg_{\m (\omega)}^{\mathrm{lin}, \mathrm{det}}(L^p)$
                  for all $\omega \in \X$, and
            \item $\EE [\m (\omega)] \leq m$.
          \end{itemize}

    \item The map $\omega \mapsto \| f - A_\omega (f) \|_{L^p}$ is measurable for every
          $f \in B^{\sigma}(\Omega)$ with $\| f \|_{B^\sigma (\Omega)} \leq 1$.
  \end{enumerate}
\end{definition}

\begin{remark*}
  Since $\Alg_m^{\mathrm{lin},\mathrm{det}}(L^p) \subset \Alg_m^{\mathrm{det}}(L^p)$
  and $\Alg_m^{\mathrm{lin}, \mathrm{MC}} \subset \Alg_m^{\mathrm{MC}}(L^p)$,
  the worst-case errors $E_p(A; B^{\sigma}(\Omega))$ and $E_p (\mathbf{A}; B^\sigma(\Omega))$
  are already defined by virtue of \Cref{def:AlgorithmsBasedOnPointSamples}.
\end{remark*}

We now aim to prove that linear algorithms (deterministic or randomized) for reconstructing
Barron functions from point samples are subject to the curse of dimensionality.
For doing so, we start with the following lemma.

\begin{lemma}\label{lem:CosineFunctionsAreBarron}
  Let $d \in \N$ and $\sigma > 0$ be arbitrary and let $\Omega = [-\frac{1}{2}, \frac{1}{2}]^d$.
  For $\eta \in \N_0^d$, define
  \[
    f_\eta : \quad
    \Omega \to \R, \quad
    f_\eta(x) = \begin{cases}
                    1,                                                 & \text{if } \eta = 0, \\
                    \sqrt{2} \, \cos(2 \pi \langle \eta, x \rangle),   & \text{otherwise} .
                  \end{cases}
  \]
  Then the family $(f_\eta)_{\eta \in \N_0^d} \subset L^2(\Omega)$ is an orthonormal system.

  Furthermore, for every $\eta \in \N_0^d$, we have
  \[
    \| f_\eta \|_{B^\sigma(\Omega)}
    \leq C \cdot (1 + |\eta|)^{\sigma}
  \]
  for a constant $C = C(d, \sigma) > 0$.
\end{lemma}

\begin{proof}
  First, we note for $\eta \in \Z^d$ that
  $\cos(2 \pi \langle \eta ,x \rangle) = \operatorname{Re} e^{2 \pi i \langle \eta, x \rangle}$
  and thus
  \(
    \int_{[-\frac{1}{2}, \frac{1}{2}]^d}
      \cos(2 \pi \langle \eta, x \rangle)
    \, d x
    = \operatorname{Re}
      \int_{[-\frac{1}{2}, \frac{1}{2}]^d}
        e^{2 \pi i \langle \eta, x \rangle}
      \, d x
    = \delta_{\eta,0}
    .
  \)
  Next, we recall the well-known formula
  $\cos (x) \cos(y) = \frac{1}{2} \big( \cos(x-y) + \cos(x+y) \big)$.
  Using it, we see for $\eta,\zeta \in \N_0^d$ that
  \begin{align*}
    & \int_{[-\frac{1}{2},\frac{1}{2}]^d}
        \cos(2 \pi \langle \eta, x \rangle) \cos(2 \pi \langle \zeta, x \rangle)
      \, d x \\
    & = \frac{1}{2}
        \int_{[-\frac{1}{2},\frac{1}{2}]^d}
           \cos(2 \pi \langle \eta - \zeta, x \rangle) + \cos(2 \pi \langle \eta + \zeta, x \rangle) 
        \, d x \\
    & = \frac{1}{2} (\delta_{\eta - \zeta, 0} + \delta_{\eta + \zeta, 0})
      = \frac{1}{2} (\delta_{\eta,\zeta} + \delta_{\eta,0}\delta_{\zeta,0})
    ,
  \end{align*}
  and thus
  \[
    \int_{[-\frac{1}{2},\frac{1}{2}]^d}
      \cos(2 \pi \langle \eta, x \rangle) \cos(2 \pi \langle \zeta, x \rangle)
    \, d x
    = \begin{cases}
        1,           & \text{if } \eta = \zeta = 0, \\
        \frac{1}{2}, & \text{if } \eta = \zeta \neq 0, \\
        0 ,          & \text{if } \eta \neq \zeta .
      \end{cases}
  \]
  This implies (the well-known fact) that $(f_\eta)_{\eta \in \N_0^d} \subset L^2(\Omega)$
  is an orthonormal system.

  Next, fix $\varphi \in C_c^\infty(\R^d)$ with $\varphi \equiv 1$ on $[-\frac{1}{2}, \frac{1}{2}]^d$
  and with $\varphi \geq 0$, and define $w : \R^d \to (0,\infty)$ by $w(\xi) := (1 + |\xi|)^{\sigma}$.
  As seen in \Cref{eq:WeightSubmultiplicative}, we have $w(\xi + \eta) \leq w(\xi) \cdot w(\eta)$
  for all $\xi, \eta \in \R^d$.
  Note that $g_\eta : \R^d \to \R$ defined by
  $g_\eta (x) = \cos(2 \pi \langle \eta, x \rangle) \cdot \varphi(x)$ satisfies
  \[
    g_\eta (x)
    = \frac{1}{2}
      \bigl(e^{2 \pi i \langle \eta, x \rangle} + e^{-2 \pi i \langle \eta, x \rangle}\bigr)
      \cdot \varphi(x)
    ,
  \]
  which implies that
  \(
    \Fourier g_\eta (\xi)
    = \frac{1}{2} \bigl( \widehat{\varphi} (\xi - \eta) + \widehat{\varphi}(\xi + \eta) \bigr)
    ,
  \)
  by elementary properties of the Fourier transform.
  Therefore,
  \begin{align*}
    & \int_{\R^d}
        w(\xi) \cdot |\Fourier g_\eta (\xi)|
      \, d \xi \\
    & \leq \frac{1}{2}
           \bigg(
             \int_{\R^d}
               |\widehat{\varphi}(\xi - \eta)| \cdot w(\xi - \eta + \eta)
             \, d \xi
             + \int_{\R^d}
                 |\widehat{\varphi}(\xi + \eta)| \cdot w(\xi + \eta - \eta)
               \, d \xi
           \bigg) \\
    & \leq w(\eta) \cdot \| \widehat{\varphi} \|_{L_w^1}
      \leq C_1 \cdot (1 + |\eta|)^{\sigma}
  \end{align*}
  for a suitable constant $C_1 = C_1(d,\sigma) > 0$.
  Since we have by Fourier inversion that $g_\eta = \Fourier^{-1} \Fourier g_\eta$
  and since $g_\eta (x) = \cos(2 \pi \langle \eta, x \rangle)$ for $x \in \Omega$,
  we see by definition of the Barron norm that
  \(
    \| \cos(2 \pi \langle \eta, \bullet \rangle) \|_{B^\sigma(\Omega)}
    \leq C_1 \cdot (1 + |\eta|)^{\sigma}.
  \)
  This easily implies that $\| f_\eta \|_{B^{\sigma}} \leq C_2 \cdot (1 + |\eta|)^{\sigma}$
  for a suitable constant $C_2 = C_2(d,\sigma) > 0$.
\end{proof}

\begin{theorem}\label{thm:LinearLowerBound}
  Let $d \in \N$ and $\sigma > 0$ be arbitrary and set $\Omega := [-\frac{1}{2}, \frac{1}{2}]^d$.
  Then there exists a constant $C = C(d,\sigma) > 0$ such that for every
  $p \in [2,\infty]$ and every $\mathbf{A} \in \Alg_m^{\mathrm{lin}, \mathrm{MC}}(L^p)$, we have
  \[
    E_p (\mathbf{A}; B^{\sigma}(\Omega))
    \geq C \cdot m^{-\sigma / d}
    .
  \]
\end{theorem}

\begin{proof}
  Since $\lebesgue(\Omega) = 1$, it is an easy consequence of Hölder's inequality
  (or Jensen's inequality) that $L^p (\Omega) \hookrightarrow L^2(\Omega)$
  with $\| \bullet \|_{L^2 (\Omega)} \leq \| \bullet \|_{L^p(\Omega)}$;
  therefore, it is enough to consider the case $p = 2$.
  With $f_\eta$, $\eta \in \N_0^d$ as defined in \Cref{lem:CosineFunctionsAreBarron},
  it is an easy consequence of that lemma that there exists a constant $C_1 = C_1(d,\sigma) > 0$
  satisfying $\| f_\eta \|_{B^\sigma(\Omega)} \leq C_1 \cdot (1 + \| \eta \|_{\ell^\infty})^\sigma$
  for all $\eta \in \N_0^d$.
  The remainder of the proof is divided into two steps.

  \medskip{}

  \noindent
  \textbf{Step~1 (average-case lower bound for linear \emph{deterministic} algorithms):}
  Let $m \in \N$ be arbitrary.
  Let $C_2 := 4^{\sigma/d} 2^\sigma C_1$ and $\gamma := C_2 \cdot m^{\sigma / d}$.
  We first claim that
  \[
    I_\gamma
    := \big\{
         \eta \in \N_0^d
         \colon
         C_1 \cdot (1 + \| \eta \|_{\ell^\infty})^\sigma \leq \gamma
       \big\}
  \]
  satisfies $|I_\gamma| \geq 4 m$.
  To see this, note that if $\eta \in \N_0^d \setminus \{ 0 \}$
  satisfies
  \[
    \| \eta \|_{\ell^\infty}
    \leq \big\lfloor \gamma^{1/\sigma} / (2 C_1^{1/\sigma}) \big\rfloor
    ,
  \]
  then
  \[
    C_1 \cdot (1 + \| \eta \|_{\ell^\infty})^\sigma
    \leq C_1 \cdot \bigl( 2 \| \eta \|_{\ell^\infty} \bigr)^\sigma
    \leq C_1 \cdot \bigl( \gamma^{1/\sigma} / C_1^{1/\sigma} \bigr)^\sigma
    =    \gamma
    ,
  \]
  and this also remains true if $\eta = 0$, since $\gamma \geq C_2 \geq C_1$.
  Therefore, we see
  \begin{align*}
    |I_\gamma|
    & \geq \Big|
            \Big\{
              \eta \in \N_0^d
              \colon
              \| \eta \|_{\ell^\infty}
              \leq \big\lfloor \gamma^{1/\sigma} / (2 C_1^{1/\sigma}) \big\rfloor
            \Big\}
           \Big| \\
    & = \Bigl( 1 + 2 \, \big\lfloor \gamma^{1/\sigma} / (2 C_1^{1/\sigma}) \big\rfloor \Bigr)^d \\
    & \geq \Bigl( \gamma^{1/\sigma} \big/ (2 C_1^{1/\sigma}) \Bigr)^d \\
    & =    \frac{C_2^{d/\sigma} \cdot m}{2^d C_1^{d/\sigma}}
      =    4 m ,
  \end{align*}
  as claimed.

  Now, let $A \in \Alg_{2 m}^{\mathrm{lin},\mathrm{det}}(L^2)$ be arbitrary;
  this means that there exists a linear map $T : \R^{2 m} \to L^2(\Omega)$ and a certain map
  $\Phi : C(\Omega) \to \R^{2 m}$ satisfying $A = T \circ \Phi$.
  Since $T : \R^{2 m} \to L^2(\Omega)$ is linear, we have $\dim (\operatorname{range} T) \leq 2 m$,
  so that we can choose orthonormal functions $g_1,\dots,g_{2 m} \in L^2(\Omega)$ satisfying
  $\operatorname{range} (T) \subset \linspan \{ g_1,\dots,g_{2 m} \} =: V$.
  Thus, denoting by $\pi_V$ the orthogonal projection onto $V$,
  we have for every $f \in B^{\sigma}(\Omega)$ that
  \begin{equation}
    \| f - A (f) \|_{L^2(\Omega)}
    = \| f - T(\Phi(f)) \|_{L^2(\Omega)}
    \geq \| f - \pi_V f \|_{L^2(\Omega)}
    .
    \label{eq:LinearErrorLowerBoundedByProjectionError}
  \end{equation}

  Now, note (as in the proof of \cite[Lemma~6]{BarronUniversalApproximationBoundsForSigmoidal})
  because of $|I_\gamma| \geq 4 m$ that
  \begin{align*}
    \avsum_{\eta \in I_\gamma}
      \| \pi_V f_\eta \|_{L^2}^2
    & = \frac{1}{|I_\gamma|}
        \sum_{\eta \in I_\gamma}
          \sum_{j=1}^{2 m}
            \langle f_\eta, g_j \rangle_{L^2(\Omega)} \\
    & = \frac{1}{|I_\gamma|}
        \sum_{j=1}^{2 m}
          \sum_{\eta \in I_\gamma}
            \langle f_\eta, g_j \rangle_{L^2(\Omega)} \\
    & \overset{(\ast)}{\leq} \frac{1}{|I_\gamma|}
                               \sum_{j=1}^{2 m}
                                 \| g_j \|_{L^2(\Omega)}^2
      = \frac{2m}{|I_\gamma|}
      \leq \frac{1}{2}
      .
  \end{align*}
  Here, the step marked with $(\ast)$ is justified by \emph{Bessel's inequality}
  (see e.g.\ \cite[Theorem~5.26]{FollandRA})
  Now, since $\| f_\eta - \pi_V f_\eta \|_{L^2} \leq \| f_\eta \|_{L^2} \leq 1$, this implies
  \begin{equation}
    \begin{split}
      \avsum_{\eta \in I_\gamma}
        \| (f_\eta / \gamma) - A(f_\eta / \gamma) \|_{L^2(\Omega)}
      & \geq \avsum_{\eta \in I_\gamma}
               \| (f_\eta / \gamma) - \pi_V (f_\eta / \gamma) \|_{L^2(\Omega)} \\
      & =    \frac{1}{\gamma}
             \avsum_{\eta \in I_\gamma}
               \| f_\eta - \pi_V f_\eta \|_{L^2(\Omega)} \\
      & \geq \frac{1}{\gamma}
             \avsum_{\eta \in I_\gamma}
               \| f_\eta - \pi_V f_\eta \|_{L^2(\Omega)}^2 \\
      & \overset{(\ast)}{=}
             \frac{1}{\gamma}
             \avsum_{\eta \in I_\gamma}
               \big( \| f_\eta \|_{L^2}^2 - \| \pi_V f_\eta \|_{L^2}^2 \big) \\
      & =    \frac{1}{\gamma}
             \Big( 1 - \avsum_{\eta \in I_\gamma} \| \pi_V f_\eta \|_{L^2}^2 \Big) \\
      & \geq \frac{1}{2 \gamma}
        =    \bigl(2 \, C_2 \cdot m^{\sigma / d}\bigr)^{-1}
        =:   C_3 \cdot m^{-\sigma / d}
        .
    \end{split}
    \label{eq:LinearReconstructionAverageCaseLowerBound}
  \end{equation}
  Here, the step marked with $(\ast)$ used that $f_\eta - \pi_V f_\eta = \pi_{V^\perp} f_\eta$
  is orthogonal to $\pi_V f_\eta$ and thus
  \(
    \| f_\eta \|_{L^2}^2
    = \| (f_\eta - \pi_V f_\eta) + \pi_V f_\eta \|_{L^2}^2
    = \| f_\eta - \pi_V f_\eta \|_{L^2}^2 + \| \pi_V f_\eta \|_{L^2}^2
  \)
  and therefore
  \(
    \| f_\eta - \pi_V f_\eta \|_{L^2}^2
    = \| f_\eta \|_{L^2}^2 - \| \pi_V f_\eta \|_{L^2}^2
    = 1 - \| \pi_V f_\eta \|_{L^2}^2
    .
  \)
  Here, the last step used that $(f_\eta)_{\eta \in \N_0^d}$ is an orthonormal
  system by construction and thus $\| f_\eta \|_{L^2} = 1$; see \Cref{lem:CosineFunctionsAreBarron}.

  Note that \Cref{eq:LinearReconstructionAverageCaseLowerBound}
  holds for all $A \in \Alg_{2 m}^{\mathrm{lin},\mathrm{det}}(L^2)$ and that
  each $h_\eta := f_\eta / \gamma$ satisfies
  \(
    \| h_\eta \|_{B^\sigma(\Omega)}
    = \frac{1}{\gamma} \| f_\eta \|_{B^\sigma(\Omega)}
    \leq \frac{1}{\gamma} C_1 \cdot (1 + \|\eta\|_{\ell^\infty})^{\sigma}
    \leq \frac{1}{\gamma} \gamma
    = 1
  \)
  for all $\eta \in I_\gamma$.

  \medskip{}

  \noindent
  \textbf{Step~2 (completing the proof):}
  The remainder of the proof is similar to Step~3 in the proof of \Cref{thm:LowerBound}.
  Let $\mathbf{A} = (A_\omega)_{\omega \in \mathbf{X}} \in \Alg_m^{\mathrm{lin},\mathrm{MC}}(L^2)$
  be arbitrary.
  This means that there exist a probability space $(\mathbf{X}, \PP)$ and a measurable map
  $\mathbf{m} : \mathbf{X} \to \N$ such that $\mathbf{A} = (A_\omega)_{\omega \in \mathbf{X}}$
  with $A_\omega \in \Alg^{\mathrm{lin},\mathrm{det}}_{\mathbf{m}(\omega)} (L^2)$
  for all $\omega \in \mathbf{X}$ and such that $\EE [\mathbf{m}(\omega)] \leq m$.
  Let $\mathbf{X}_0 := \{ \omega \in \mathbf{X} \colon \mathbf{m}(\omega) \leq 2 m \}$.
  Exactly as in Step~3 of the proof of \Cref{thm:LowerBound},
  one sees that $\PP(\mathbf{X}_0) \geq \frac{1}{2}$.

  Now, recalling the definition of $E_2(\mathbf{A}; B^\sigma(\Omega))$
  and that $\| h_\eta \|_{B^\sigma} \leq 1$ for $\eta \in I_\gamma$,
  and noting that $A_\omega \in \Alg_m^{\mathrm{lin},\mathrm{det}}(L^2)$ for $\omega \in \mathbf{X}_0$,
  so that \Cref{eq:LinearReconstructionAverageCaseLowerBound} is applicable, we see
  \begin{align*}
    E_2 (\mathbf{A}; B^\sigma(\Omega))
    & = \sup_{\| f \|_{B^\sigma(\Omega)} \leq 1}
          \EE \| f - A_\omega(f) \|_{L^2} \\
    & \geq \avsum_{\eta \in I_\gamma}
             \EE \| h_\eta - A_\omega(h_\eta) \|_{L^2} \\
    & =    \EE
           \bigg[
             \avsum_{\eta \in I_\gamma}
               \| h_\eta - A_\omega(h_\eta) \|_{L^2}
           \bigg] \\
    & \geq \EE
           \bigg[
             \indicator_{\mathbf{X}_0}(\omega)
             \avsum_{\eta \in I_\gamma}
               \| h_\eta - A_\omega(h_\eta) \|_{L^2}
           \bigg] \\
    & \geq \EE
           \big[
             \indicator_{\mathbf{X}_0}(\omega)
             \cdot C_3
             \cdot m^{-\sigma / d}
           \big] \\
    & \geq \frac{C_3}{2} \cdot m^{-\sigma / d}
    .
  \end{align*}
  For $C := C_3 / 2$, this implies the claimed lower bound, since we noted at the beginning
  of the proof that it is enough to treat the case $p = 2$.
\end{proof}

\appendix

\section{Postponed technical proofs}
\label{sec:Appendix}

\subsection{Proof of Lemma \ref{lem:ReconstructionOnWholeSpace}}%
\label{sub:ReconstructionOnWholeSpaceProof}

\begin{proof}
  For brevity, set $\eps := s_m (V; R)$.
  By definition, there exist $x_1,\dots,x_m \in \Omega$ and a map
  $T : \R^m \to \R$ such that $\| f - T(f(x_1),\dots,f(x_m)) \|_R \leq \eps + \delta$
  for all $f \in V$ with $\| f \|_V \leq 1$.
  Define $\Phi : V \to \R^m, f \mapsto (f(x_1),\dots,f(x_m))$.
  For each $y \in \Phi(V)$, we can choose some $g_y \in V$ satisfying
  \[
    \Phi(g_y) = y
    \qquad \text{and} \qquad
    \| g_y \|_V \leq (1 + \delta) \, \inf_{g \in V, \Phi(g) = y} \| g \|_V
    ;
  \]
  indeed, this is clear in case the infimum is positive; otherwise,
  there exists a sequence $(g_n)_{n \in \N}$ satisfying $\Phi(g_n) = y$ and $\| g_n \|_V \to 0$.
  Since $V \hookrightarrow C(\Omega)$, this implies $g_n (x_i) \xrightarrow[n\to\infty]{} 0$
  and thus $y = \lim_{n \to \infty} \Phi(g_n) = 0$, so that $g_y = 0$ is a valid choice.
  Define
  \[
    A : \quad
    \R^m \to V, \quad
    y \mapsto \begin{cases}
                g_y, & \text{if } y \in \Phi(V), \\
                0,   & \text{otherwise} .
              \end{cases}
  \]

  Now, let $f \in V$ be arbitrary and set $y := \Phi(f)$ and $C := \| f \|_V$.
  This implies
  \[
    \| g_y \|_V
    \leq (1 + \delta) \, \inf_{g \in V, \Phi(g) = y} \| g \|_V
    \leq (1 + \delta) \, \| f \|_V
    =    (1 + \delta) \, C
    .
  \]
  In case of $C = 0$, we thus have $f = g_y = 0$ and thus $\| f - A(y) \|_R = \| f - g_y \|_R = 0$,
  so that the claim is trivial; we can thus assume in what follows that $C > 0$.

  Now, set $h := \frac{f - g_y}{(2 + \delta) C}$ and note $\| h \|_{V} \leq 1$ and $\Phi(h) = 0$.
  Therefore, we see by choice of $T$ that
  \[
    \| h \|_R
    \leq \| h - T(0) \|_R + \| T(0) - 0 \|_R
    \leq 2 (\eps + \delta)
    ,
  \]
  and thus
  \begin{align*}
    \| f - A(f(x_1),\dots,f(x_m)) \|_R
    & = \| f - g_y \|_R \\
    & = (2 + \delta) C \cdot \| h \|_R \\
    & \leq 2 (\eps + \delta) \cdot (2 + \delta) \cdot C \\
    & \leq 2 (\eps + \delta) \cdot (2 + \delta) \cdot \| f \|_V
    .
  \end{align*}
  Since $\delta > 0$ can be chosen arbitrarily, this easily implies the claim.
\end{proof}

\section*{Acknowledgments}

F.\ Voigtlaender warmly thanks Jan Vybiral for helpful discussions.
Big thanks are also due to Andrei Caragea and Dae Gwan Lee,
for several very helpful discussions and for help in proofreading the paper.

F.\ Voigtlaender acknowledges support by the German Research Foundation (DFG)
in the context of the Emmy Noether junior research group VO 2594/1--1.

\footnotesize
\bibliographystyle{plain}
\bibliography{references.bib}

\begin{thebibliography}{10}

\bibitem{BarronApproximationAndEstimationBoundsForNN}
A.~R. Barron.
\newblock Approximation and estimation bounds for artificial neural networks.
\newblock {\em Mach. Learn.}, 14(1):115--133, 1994.

\bibitem{BarronUniversalApproximationBoundsForSigmoidal}
A.R. Barron.
\newblock Universal approximation bounds for superpositions of a sigmoidal
  function.
\newblock {\em IEEE Transactions on Information Theory}, 39(3):930--945, 1993.

\bibitem{DeVoreOptimalLearning}
P.~Binev, A.~Bonito, R.~DeVore, and G.~Petrova.
\newblock Optimal learning.
\newblock {\em arXiv preprint arXiv:2203.15994}, 2022.

\bibitem{CarageaPetersenVoigtlaenderNNApproximationBarronBoundary}
A.~Caragea, P.~Petersen, and F.~Voigtlaender.
\newblock Neural network approximation and estimation of classifiers with
  classification boundary in a {Barron} class.
\newblock {\em arXiv preprint arXiv:2011.09363}, 2020.

\bibitem{DungTemlyakovUllrichHyperbolicCrossApproximation}
D.~D\~{u}ng, V.~Temlyakov, and T.~Ullrich.
\newblock {\em Hyperbolic cross approximation}.
\newblock Advanced Courses in Mathematics. CRM Barcelona.
  Birkh\"{a}user/Springer, Cham, 2018.
\newblock Edited and with a foreword by Sergey Tikhonov.

\bibitem{DeVoreNonLinearApproximationByTrigonometricSum}
R.~A. DeVore and V.~N. Temlyakov.
\newblock Nonlinear approximation by trigonometric sums.
\newblock {\em J. Fourier Anal. Appl.}, 2(1):29--48, 1995.

\bibitem{WeinanERepresentationFormulasForBarronFunctions}
W.~E and S.~Wojtowytsch.
\newblock Representation formulas and pointwise properties for {Barron}
  functions.
\newblock {\em Calc. Var. Partial Differ. Equ.}, 61(2):37, 2022.
\newblock Id/No 46.

\bibitem{FollandRA}
G.~B. Folland.
\newblock {\em Real analysis}.
\newblock Pure and Applied Mathematics (New York). John Wiley \& Sons, Inc.,
  New York, second edition, 1999.

\bibitem{RauhutFoucartIntroToCS}
S.~Foucart and H.~Rauhut.
\newblock {\em A mathematical introduction to compressive sensing}.
\newblock Appl. Numer. Harmon. Anal. New York, NY: Birkh{\"a}user/Springer,
  2013.

\bibitem{GrohsVoigtlaenderNNSamplingNumbers}
P.~Grohs and F.~Voigtlaender.
\newblock Proof of the theory-to-practice gap in deep learning via sampling
  complexity bounds for neural network approximation spaces.
\newblock {\em arXiv preprint arXiv:2104.02746}, 2021.

\bibitem{HeinrichRandomApproximation}
S.~Heinrich.
\newblock Random approximation in numerical analysis.
\newblock In {\em Functional analysis ({E}ssen, 1991)}, volume 150 of {\em
  Lecture Notes in Pure and Appl. Math.}, pages 123--171. Dekker, New York,
  1994.

\bibitem{CarlsInequality}
A.~Hinrichs, A.~Kolleck, and J.~Vyb\'{\i}ral.
\newblock Carl's inequality for quasi-{B}anach spaces.
\newblock {\em J. Funct. Anal.}, 271(8):2293--2307, 2016.

\bibitem{KlusowskiBarronHighDimensionalRidgeFunctionCombinations}
J.~M. Klusowski and A.~R. Barron.
\newblock Risk bounds for high-dimensional ridge function combinations
  including neural networks.
\newblock {\em arXiv preprint arXiv:1607.01434}, 2016.

\bibitem{KlusowskiBarronApproximation}
J.~M. Klusowski and A.~R. Barron.
\newblock Approximation by combinations of {ReLU} and squared {ReLU} ridge
  functions with $\ell^1$ and $\ell^0$ controls.
\newblock {\em IEEE Transactions on Information Theory}, 64(12):7649--7656,
  2018.

\bibitem{KriegSonnleitnerRandomPointsForApproximatingSobolevFunctions}
D.~Krieg and M.~Sonnleitner.
\newblock Random points are optimal for the approximation of {Sobolev}
  functions.
\newblock {\em arXiv preprint arXiv:2009.11275}, 2020.

\bibitem{SiegelUniformApproximationRatesAndNWidthsOfShallowNN}
L.~Ma, J.~W. Siegel, and J.~Xu.
\newblock Uniform approximation rates and metric entropy of shallow neural
  networks.
\newblock {\em Res. Math. Sci.}, 9(3):21, 2022.
\newblock Id/No 46.

\bibitem{MuscaluSchlagHA}
C.~Muscalu and W.~Schlag.
\newblock {\em Classical and multilinear harmonic analysis. {V}ol. {I}}, volume
  137 of {\em Cambridge Studies in Advanced Mathematics}.
\newblock Cambridge University Press, Cambridge, 2013.

\bibitem{UllrichNewUpperBoundForSamplingNumbers}
N.~Nagel, M.~Sch\"{a}fer, and T.~Ullrich.
\newblock A new upper bound for sampling numbers.
\newblock {\em Found. Comput. Math.}, 22(2):445--468, 2022.

\bibitem{PetersenVoigtlaenderOptimalLearningBarronBoundary}
P.~Petersen and F.~Voigtlaender.
\newblock Optimal learning of high-dimensional classification problems using
  deep neural networks.
\newblock {\em arXiv preprint arXiv:2112.12555}, 2021.

\bibitem{UllrichRKHSSamplingRecoveryLInfty}
K.~Pozharska and T.~Ullrich.
\newblock A note on sampling recovery of multivariate functions in the uniform
  norm.
\newblock {\em SIAM J. Numer. Anal.}, 60(3):1363--1384, 2022.

\bibitem{SiegelEntropyAndWidthOfShallowNN}
J.~W. Siegel and J.~Xu.
\newblock Sharp bounds on the approximation rates, metric entropy, and
  $n$-widths of shallow neural networks.
\newblock {\em arXiv preprint arXiv:2101.12365}, 2021.

\end{thebibliography}

\end{document}